\documentclass[10pt, a4paper,english]{amsart}
\usepackage[utf8]{inputenc}
\usepackage[english]{babel}
\usepackage{amsmath,amsthm,amssymb,tikz}
\usepackage{bm} 
\usepackage{hyperref, enumerate}
\usepackage{graphicx}
\usepackage[all,2cell]{xy}
\usepackage{comment}
\usepackage{tikz-cd}
\usepackage{tikz}
\usepackage{adjustbox}

\newtheorem{theorem}{Theorem}[section]
\newtheorem{lemma}[theorem]{Lemma}
\newtheorem{proposition}[theorem]{Proposition}
\newtheorem{corollary}[theorem]{Corollary}
\theoremstyle{definition}
\newtheorem{definition}[theorem]{Definition}

\newtheorem{example}[theorem]{Example}
\newtheorem{assumption}[theorem]{Assumption}
\theoremstyle{remark}

\newtheorem{remark}{Remark}
\newtheorem{question}{Question}
\newtheorem{claim}{Claim}
\newcommand{\alg}[1]{\mathbf{#1}}
\newcommand{\cate}[1]{{\mathsf{#1}}}	
\newcommand{\var}[1]{\mathcal{#1}}
\newcommand{\func}[1]{\mathsf{#1}}
\newcommand{\lucas}{\text{\bf\L}}
\newcommand{\Z}{\mathbb{Z}}

\newcommand{\Stone}{\mathsf{Stone}}
\newcommand{\Set}{\mathsf{Set}}
\newcommand{\BA}{\mathsf{BA}}
\newcommand{\CABA}{\mathsf{CABA}}
\newcommand{\CAA}{\mathsf{CA}\var{A}}
\newcommand{\FLew}{\mathsf{FL}_{ew}}
\newcommand{\variety}[1]{\mathbb{H}\mathbb{S}\mathbb{P}(#1)}
\newcommand{\qvariety}[1]{\mathbb{I}\mathbb{S}\mathbb{P}(#1)}
\newcommand{\StoneL}{\mathsf{Stone}_{\alg{L}}}
\newcommand{\SetL}{\mathsf{Set}_{\alg{L}}}
\newcommand{\val}[1]{\mathbf{#1}}
\newcommand{\U}{\mathsf{U}}
\newcommand{\Vtop}{\mathsf{V}^\top}
\newcommand{\Vbot}{\mathsf{V}^\bot}
\newcommand{\C}{\mathsf{C}}
\newcommand{\M}{\mathfrak{P}}
\newcommand{\B}{\mathfrak{S}}
\newcommand{\Q}{\mathsf{Q}}
\newcommand{\I}{\mathsf{I}}
\newcommand{\V}{\mathsf{V}}
\newcommand{\K}{\mathsf{K}}
\newcommand{\Ind}{\mathsf{Ind}}
\newcommand{\Pro}{\mathsf{Pro}}
\newcommand{\pr}{\mathsf{pr}}
\newcommand{\im}{\mathsf{im}}
\newcommand{\id}{\mathsf{id}}

\author{Alexander Kurz}
\address{Chapman University, 1 University Drive, 92866 Orange, California, USA}
\email{akurz@chapman.edu}

\author{Wolfgang Poiger}
\address{University of Luxembourg, 6 Avenue de la Fonte, L-4364 Esch-sur-Alzette, Luxembourg}
\email{wolfgang.poiger@uni.lu \\ bruno.teheux@uni.lu}

\author{Bruno Teheux}

\title[New perspectives on semi-primal varieties]{New perspectives on semi-primal varieties}



\begin{document}
\begin{abstract}
We study varieties generated by semi-primal lattice-expansions by means of category theory. We provide a new proof of the Keimel-Werner topological duality for such varieties and, using similar methods, establish its discrete version. We describe multiple adjunctions between the variety of Boolean algebras and the variety generated by a semi-primal lattice-expansion, both on the topological side and explicitly algebraic. In particular, we show that the Boolean skeleton functor has two adjoints, both defined by taking certain Boolean powers, and we identify properties of these adjunctions which fully characterize semi-primality of an algebra. Lastly, we give a new characterization of canonical extensions of algebras in semi-primal varieties in terms of their Boolean skeletons. 
\end{abstract}
\maketitle
\section{Introduction}
Primality and its variations are classical topics in universal algebra which were prominently studied during the second half of the 20th century \cite{Quackenbush1979, Werner1978,Burris1992}. During the 1950s, Foster introduced primal algebras in his generalized `Boolean' theory of universal algebras \cite{Foster1953a, Foster1953b}. Generalizing functional completeness of the two-element Boolean algebra, an algebra $\alg{P}$ is \emph{primal} if every operation $f\colon P^n \to P$ is term-definable in $\alg{P}$. The intuition that a primal algebra $\alg{P}$ is `close to' the two-element Boolean algebra $\alg{2}$ was confirmed by Hu's theorem \cite{Hu1969,Hu1971}, which states that a variety $\var{V}$ is categorically equivalent to the variety $\BA$ of Boolean algebras (generated by $\alg{2}$) if and only if $\var{V}$ is generated by a primal algebra $\alg{P}\in \var{V}$. 

In 1964, Foster and Pixley introduced the first variation of primality, which they called \emph{semi-primality} \cite{FosterPixley1964a}. Unlike primal algebras, a semi-primal algebra may have proper subalgebras. Accordingly, in a semi-primal algebra $\alg{L}$, we only require the operations $f\colon L^n\to L$ which \emph{preserve subalgebras} to be term-definable in $\alg{L}$. Semi-primal varieties (that is, varieties of the form $\variety{\alg{L}}$ where $\alg{L}$ is semi-primal) are well-understood from the viewpoint of `classical' universal algebraic structure theory \cite{FosterPixley1964a, FosterPixley1964b, Foster1967,MooreYaqub1968} as well as from the viewpoint of duality theory \cite{KeimelWerner1974,ClarkDavey1998}. From the perspective of category theory, semi-primal varieties were classified up to Morita equivalence in \cite{Bergman1996} - however, this is done using purely algebraic tools based on \cite{McKenzie1996}. In this paper, we further advance the category theoretical study of semi-primality by putting a semi-primal variety $\var{A}$ in relationship with other varieties, in particular with the primal variety $\BA$. Although Hu's theorem implies that the varieties are usually not categorically equivalent, we demonstrate that, nevertheless, there is a rich relationship between $\var{A}$ and $\BA$. In particular we explicate the intuition that semi-primal algebras are still `close to' the two-element Boolean algebra.

More specifically, we investigate multiple adjunctions between $\BA$ and the variety $\var{A}$ generated by a semi-primal algebra $\alg{L}$ with an \emph{underlying bounded lattice} (see Assumption \ref{assumption}). For one, this assumption yields a useful characterization of semi-primality via certain unary terms (see Proposition \ref{SP-char-Ts}) which we prominently use. Furthermore, since $\alg{L}$ has no one-element subalgebras, the dual category of $\var{A}$ has a particularly simple description (see Definition \ref{def:StoneL}). Apart from these advantages, the restriction to lattice-based algebras is motivated by the connection to many-valued logic. If we consider $\alg{L}$ as an algebra of propositional truth-degrees, an underlying bounded lattice is a reasonable assumption. For example, Maruyama \cite{Maruyama2012} generalized Jónsson-Tarski duality to modal extensions of semi-primal algebras with bounded lattice reducts. We plan to demonstrate applications of our results to many-valued (coalgebraic) modal logic in subsequent work. Since, in addition, there are already plenty of examples of such algebras (see Subsection \ref{subs:semi-primal-examples}), it is reasonable to stick to this framework. 

Although we were mainly motivated by questions arising in logic, we particularly hope that this paper will be of interest to algebraists interested in category theory as well as to category theorists interested in universal algebra. Let us point out that, in this paper, the category theoretical approach to universal algebra is different from other common ones via Lawvere theories or monads (these are well-exposed in \cite{HylandPower2007}).
Indeed, this paper is not about reformulating and generalizing algebraic concepts into categorical language, but rather to apply category theory as a tool to gain new insight into a concrete topic in universal algebra. For example, the fact that the variety $\var{A}$ is the completion of the full subcategory of its finite members $\var{A}^\omega$ under filtered colimits (i.e., $\var{A} \simeq \Ind(\var{A}^\omega)$) can be helpful to make the step from finite to infinite, for example to extend functors defined on $\var{A}^\omega$ to the full variety $\var{A}$ in a canonical way. Motivated by \cite{Johnstone1982}, we furthermore use this fact to give a new proof of the semi-primal duality \cite{KeimelWerner1974,ClarkDavey1998} by lifting the corresponding finite duality (see Theorem \ref{thm:ProSetL=StoneL} and Theorem \ref{thm:IndSetL}). By replacing $\Ind(\var{A}^\omega)$ by $\Pro(\var{A}^\omega)$, the closure under cofiltered limits, we prove the discrete version of the duality (resembling the duality between $\Set$ and the category $\CABA$ of complete atomic Boolean algebras) in a similar manner.             

The paper is organized as follows. In Section \ref{sec:semiprimalalgs} we recall well-known results about semi-primal algebras and the varieties they generate. In particular, we discuss semi-primal bounded-lattice expansions and provide examples thereof. In Section \ref{sec:semi-primal duality} we describe the topological duality for semi-primal algebras and, as mentioned above, provide an alternative proof for it. Arguably the most important results of the paper are exposed in Section \ref{sec:adcuntions}, where we describe a chain of four adjoint functors between $\var{A}$ and $\BA$ (see Figure \ref{fig:adjunctions}). Most prominently, the adjunction $\B \dashv \M$ is described in detail, first via duality and then explicitly algebraically (see Theorem \ref{thm:MadjointB}). The \emph{Boolean skeleton} $\B\colon \var{A}\to \BA$ has, for example, been known for $\mathsf{MV}_n$-algebras \cite{Cignoli2000} and was generalized to arbitrary semi-primal bounded lattice expansions by Maruyama \cite{Maruyama2012}. Its right-adjoint $\M\colon \BA \to \var{A}$ relies on the construction of a \emph{Boolean power} \cite{Burris1975}, a certain Boolean product \cite{BurrisSankappanavar1981} which was already introduced for arbitrary finite algebras in Foster's original paper on primality \cite{Foster1953a}. In the case where $\alg{L}$ is primal, we retrieve a concrete categorical equivalence witnessing Hu's theorem (see Corollary \ref{cor:hu}). We proceed to investigate the \emph{subalgebra adjunctions}, which exist for each subalgebra $\alg{S} \leq \alg{L}$. We manage to trace them back to the adjunction $\B \dashv \M$ after taking an appropriate inclusion/quotient (see Theorem \ref{thm:subadjunction}). In particular, we illustrate why the subalgebra adjunction $\Q \dashv \I$ corresponding to the smallest subalgebra of $\alg{L}$ is of special interest. Indeed, towards the end of Section \ref{sec:adcuntions} we also show that the existence of an adjoint situation resembling $\I \dashv \B \dashv \M$ fully characterizes semi-primality of a lattice-based algebra (see Theorem \ref{thm:adjunctionSemiCharact}). Building on the results of Section \ref{sec:adcuntions}, in Section \ref{sec:canext} we prove the above-mentioned discrete duality for $\Pro(\var{A}^\omega)$. It is well-known that the algebras in this category correspond to the \emph{canonical extensions} \cite{GehrkeJonsson2004,DaveyPriestley2012} of algebras in $\var{A}$. Notably, we show that these canonical extensions may be characterized almost purely in terms of their Boolean skeletons (see Theorem \ref{thm:ProACAA}). Lastly we connect Sections \ref{sec:adcuntions} and \ref{sec:canext} by describing an analogue of the Stone-Čech compactification in our setting (see Proposition \ref{prop:StoneCech}).   

We summarize our results schematically in Section \ref{sec:conclusion} (see Figure \ref{fig:summary}). In addition to the logical ramifications already mentioned, we believe that there are more potential ways to follow up our results. In particular, we hope to inspire further research in universal algebra through the lens of category-theory. Some open questions directly related to the content of this paper are also collected in Section \ref{sec:conclusion}. 
\section{Semi-primal algebras and the varieties they generate}\label{sec:semiprimalalgs}
In the 1950s, Foster introduced the concept of primality in \cite{Foster1953a, Foster1953b}, generalizing functional completeness of the two-element Boolean algebra $\alg{2}$. 
A finite algebra $\alg{L}$ is called \emph{primal} if, for all $n\geq 1$, every function $f\colon L^n \to L$ is term-definable in $\alg{L}$. 
Besides the two-element Boolean algebra $\alg{2}$, the $(n+1)$-element Post chain $\alg{P}_n$ and the field of prime order $\Z/p\Z$ with $0$ and $1$ as constants are some famous examples of primal algebras.

Using Stone duality, Hu \cite{Hu1969,Hu1971} showed that a variety $\var{A}$ is generated by a primal algebra (in other words, $\mathcal{A} = \variety{\alg{L}}$ for some primal algebra $\alg{L}$) if and only if $\var{A}$ is categorically equivalent to the variety of Boolean algebras $\BA$ (see also \cite{Porst2000} for a treatment using Lawvere theories). Of course we don't expect any more meaningful category theoretical results about the relationship between $\mathcal{A}$ and $\BA$ in this case. One purpose of this paper is to demonstrate that, in contrast, such results \emph{do} arise as soon as we assume that $\alg{L}$ is  semi-primal.   
\subsection{Characterizations of semi-primality}\label{subsection: semi-primal-algebras}
Since Foster's original work, many variations of primality have been introduced (for overviews see, \emph{e.g.}, \cite{Quackenbush1979, KaarliPixley2001}). Among them, intuitively speaking, semi-primality seems to still be rather close to primality (a central theme of this paper is to show why this intuition is justified). In a slogan: \emph{semi-primal algebras are like primal algebras which allow subalgebras}.     

Note that a primal algebra $\alg{L}$ does not have any proper subalgebra $\alg{S} \lneqq \alg{L}$. Otherwise, picking any $s\in S$ and $\ell \in L{\setminus} S$, no function $f: L \to L$ with $f(s) = \ell$ can possibly be term-definable.  
 
Semi-primality, introduced by Foster and Pixley in 1964 (see \cite{FosterPixley1964a}) does not impose this restriction. Recall that a function  $f \colon L^n \to L$ \emph{preserves subalgebras} if $f(a_1, \ldots, a_n)$ is in the subalgebra  generated by $\{a_1, \ldots, a_n\}$ for any choice of $a_1, \ldots, a_n \in L$. Clearly, if a function is term-definable, then it preserves subalgebras. In semi-primal algebras, the converse also holds. 

\begin{definition}
A finite algebra $\alg{L}$ is \emph{semi-primal} (sometimes also called \emph{subalgebra-primal}) if for every $n\geq 1$, every function $f\colon L^n\to L$ which preserves subalgebras is term-definable in $\alg{L}$.
\end{definition}
For example, the field of prime-order $\mathbb{Z}/p\mathbb{Z}$ with only $0$ as constant is semi-primal but not primal anymore - it now has $\{ 0 \}$ as proper subalgebra. More interesting examples are described in detail in Subsection \ref{subs:semi-primal-examples}. 
In the following we recall two well-known equivalent characterizations of semi-primality. The first one is based on the ternary discriminator term and the second one is based on the existence of a majority term. 

First we recall the characterization of semi-primal algebras as special instances of \emph{discriminator algebras}. These are the algebras in which the \emph{ternary discriminator}
$$ t(x,y,z) = \begin{cases}
z & \text{ if } x = y \\
x & \text{ if } x \neq y
\end{cases} $$ 
is term-definable. Finite discriminator algebras are also called \emph{quasi-primal}. 

An \emph{internal isomorphism} of $\alg{L}$ is an isomorphism $\varphi\colon \alg{S}_1 \to \alg{S}_2$ between any two (not necessarily distinct) subalgebras $\alg{S}_1$ and $\alg{S}_2$ of $\alg{L}$. For example, if $\alg{S} \leq \alg{L}$ is a subalgebra, then the identity $id_S$ is an internal isomorphism of $\alg{L}$. In semi-primal algebras, there are no other internal isomorphisms. 

\begin{proposition}{\cite[Theorem 3.2.]{Pixley1971}}\label{SP-char-discriminator}
A finite algebra $\alg{L}$ is semi-primal if and only if it is quasi-primal and the only internal isomorphisms of $\alg{L}$ are the identities on subalgebras of $\alg{L}$.   
\end{proposition} 
Secondly, we recall the characterization of semi-primality based on a majority term, which can be useful to generate examples (see, for example, \cite{DaveySchumann1991}). Recall that a \emph{majority term} is a ternary term $m(x,y,z)$ satisfying 
$$m(x,x,y) = m(x,y,x) = m(y,x,x) = x.$$ 
In particular, every lattice $\alg{L} = (L,\wedge,\vee)$ has a majority term given by the \emph{median}
$$m(x,y,z) = (x\wedge y) \vee (x\wedge z) \vee (y\wedge z).$$ 
\begin{proposition}{\cite[Theorem 7.2.]{BakerPixley1975}}\label{SP-char-squaresubalg}
A finite algebra $\alg{L}$ is semi-primal if and only if it has a majority term and every subalgebra of $\alg{L}^2$ is either the direct product of two subalgebras or the diagonal of a subalgebra of $\alg{L}$.
\end{proposition}

The structure of semi-primal varieties was already well-studied in the original work by Foster and Pixley during the 1960s. To stay self-contained, we recall some results about these varieties which will be of use for us later.   
  
\begin{proposition}{\cite[Theorem 4.2]{FosterPixley1964a}}
The variety $\var{A}$ generated by a semi-primal algebra $\alg{L}$ coincides with the quasi-variety generated by $\alg{L}$, that is $\var{A} = \qvariety{\alg{L}}.$
\end{proposition} 
In addition to the characterizations above, there is a nice characterization of semi-primality of $\alg{L}$ in terms of $\var{A}$. Recall that a variety is called \emph{arithmetical} if it is congruence distributive and congruence permutable.
\begin{proposition}{\cite[Theorem 3.1]{FosterPixley1964b}}\label{SP-char-simplearithmetic}
A finite algebra $\alg{L}$ is semi-primal if and only if the variety generated by $\alg{L}$ is arithmetical, every subalgebra of $\alg{L}$ is simple, and the only internal isomorphisms of $\alg{L}$ are the identities of subalgebras.  
\end{proposition}

\begin{remark}
Together with Proposition \ref{prop:homsareprojections} this implies that if $\alg{L}$ is semi-primal, then the collection of subalgebras $\mathbb{S}(\alg{L})$ considered as a \emph{subcategory} of the variety generated by $\alg{L}$, forms a lattice, ordered under inclusion. \hfill $\blacksquare$   
\end{remark} 
 
The finite members of $\var{A}$ are particularly well-behaved. For notation, given a concrete category $\cate{C}$, we use $\cate{C}^\omega$ to denote the full subcategory of $\cate{C}$ generated by its finite members. In particular, if $\var{A}$ is a variety, we use $\var{A}^\omega$ to denote the category of finite algebras in $\var{A}$.

\begin{proposition}{\cite[Theorem 7.1]{FosterPixley1964a}}\label{Afin = PS(L)}
Let $\var{A}$ be the variety generated by a semi-primal algebra $\alg{L}$. Every finite algebra $\alg{A}\in\var{A}^\omega$ is isomorphic to a direct product of subalgebras of $\alg{L}$. 
\end{proposition}

We add yet another characterization of semi-primality in our particular case of interest (in which the algebra is based on a bounded lattice) in the following subsection (see Proposition \ref{SP-char-Ts}).

\subsection{Semi-primal bounded lattice expansions}\label{semi-primal-lattice-expansions}
In this subsection we set the scene for the remainder of this paper. We aim to describe the relationship between the variety $\BA$ of Boolean algebras and the variety generated by a semi-primal algebra \emph{with underlying bounded lattice}. 

Under the additional assumption that $\alg{L}$ is based on a bounded lattice, there is another nice characterization of semi-primality of $\alg{L}$ which will be particularly useful for our purposes. It relies on the following unary terms.   

\begin{definition}\label{defin:Ts}
Let $\alg{L}$ be an algebra based on a bounded lattice 
$ \alg{L}^{\flat} = (L,\wedge,\vee,0,1).$ For all $\ell \in L$ we define $T_\ell\colon L\rightarrow L$ and $\tau_{\ell}\colon L\to L$ to be the characteristic function of $\{ \ell \}$ and $\{\ell' \geq \ell \}$, respectively. That is,
$$ T_\ell (x) =  \begin{cases}
1 & \text{ if } x = \ell \\
0 & \text{ if } x \neq \ell
\end{cases}
\hspace{5mm}\text{ and }\hspace{5mm} 
\tau_\ell (x) =  \begin{cases}
1 & \text{ if } x \geq \ell \\
0 & \text{ if } x \not\geq \ell.
\end{cases}$$
\end{definition}

Even though the following result is essentially an instance of the more general \cite[Theorem 4.1]{Foster1967}, we include an easy direct proof here.

\begin{proposition}{\cite[Theorem 4.1]{Foster1967}}\label{SP-char-Ts}
Let $\alg{L}$ be a finite algebra with an underlying bounded lattice. Then the following conditions are equivalent:
\begin{enumerate}
\item $\alg{L}$ is semi-primal.
\item For every $\ell\in \alg{L}$, the function $T_\ell$ is term-definable in $\alg{L}$.
\item $T_0$ is term-definable and for every $\ell\in \alg{L}$, the function $\tau_\ell$ is term-definable in $\alg{L}$.
\end{enumerate}
\end{proposition}
\begin{proof}
$(1)\Rightarrow (2)$: Since every subalgebra of $\alg{L}$ contains the set $\{0,1\}$, semi-primality of $\alg{L}$ implies that all $T_\ell$ are term-definable, since they preserve subalgebras. 

$(2)\Rightarrow (1)$: First we show that the ternary discriminator is term-definable in $\alg{L}$. Consider the term
$$ c(x,y) = \bigvee_{\ell\in L} \big((T_\ell(x) \wedge T_\ell(y)\big) ,$$
which satisfies 
$$ c(x,y) =  \begin{cases}
1 & \text{ if } x = y \\
0 & \text{ if } x \neq y
\end{cases}$$
and $d(x,y) := T_0(c(x,y))$ (note that this is the discrete metric). 
The term 
$$ t(x,y,z) = (d(x,y) \wedge x) \vee (c(x,y) \wedge z)$$
yields the ternary discriminator on $\alg{L}$. 
Now we show that the only internal isomorphisms of $\alg{L}$ are the identities of subalgebras. Let $\varphi: \alg{S}_1 \rightarrow \alg{S}_2$ be an internal isomorphism of $\alg{L}$ and let $s\in S_1$ be arbitrary. Then 
$$ 1 = T_{\varphi(s)}\big(\varphi(s)\big) = \varphi \big(T_{\varphi(s)}(s)\big) $$
Since $\varphi(0) = 0$ we necessarily have $T_{\varphi(s)}(s) = 1$, which is equivalent to $\varphi(s) = s$. Altogether, due to Proposition \ref{SP-char-discriminator}, we showed that $\alg{L}$ is semi-primal.  

$(2)\Rightarrow (3)$: If the $T_\ell$ are term-definable we can define 
$$\tau_\ell (x) = \bigvee_{\ell' \geq \ell} T_{\ell'} (x).$$

$(3) \Rightarrow (2)$: If $T_0$ and the $\tau_\ell$ are term-definable we can define 
$$ T_\ell(x) = \tau_{\ell} (x) \wedge \bigwedge_{\ell' > \ell} T_0\big(\tau_{\ell'} (x)\big),$$ 
which concludes the proof.
\end{proof}
\begin{remark}
In light of this result, we can turn any finite bounded lattice into a semi-primal algebra by adding $T_\ell$ as unary operation for every element $\ell \in L$. One might wonder how this differs from adding a constant symbol (i.e., a nullary operation) for every element. The difference is that adding a constant imposes the requirement that every subalgebra needs to contain the element corresponding to this constant. Thus, the algebra that results after adding all constants does not have any proper subalgebras. \hfill $\blacksquare$
\end{remark} 
We now state our main assumption, which from now on holds for the remainder of this paper.    

\begin{assumption}\label{assumption}
\textbf{The finite algebra $\alg{L}$ is semi-primal and has an underlying bounded lattice.} 
\end{assumption}
From now on, let $\var{A} := \variety{\alg{L}}$ denote  the variety generated by $\alg{L}$.
In Subsection \ref{subs:semi-primal-examples} we provide various examples of algebras satisfying Assumption \ref{assumption}.   

As noted in \cite{Maruyama2012} (where the same assumption on $\alg{L}$ is made), from the point of view of many-valued logic, semi-primal algebras make good candidates for algebras of truth-values. In this context the underlying bounded lattice is a natural minimal requirement. 
\subsection{Examples of semi-primal algebras}\label{subs:semi-primal-examples} In this subsection we collect some examples of semi-primal algebras. All of them are bounded lattice expansions (since most of them stem from many-valued logic), thus they all fit the scope of this paper (see Assumption \ref{assumption}). For other examples we refer the reader to \cite{Burris1992, Werner1978,MooreYaqub1968}.  

First, we describe several different semi-primal algebras based on finite chains. To get examples based on lattices which are not necessarily totally ordered, in Subsection~\ref{example:reslattices} (and Appendix \ref{Appendix}) we discuss semi-primal residuated lattices. In particular we describe a systematic way to identify them among the $\FLew$-algebras. Similarly, Subsection~\ref{Example: Pseudo-Logics} illustrates how to identify semi-primal algebras which need not be totally ordered among the pseudo-logics. At the end of this subsection we recall Murskiĭ's Theorem which states that, in some sense, almost all finite lattice-based algebras are semi-primal.       

\subsubsection{Chain-based algebras}\label{example:chain-based} We will describe several different ways of turning the $(n+1)$-element chain $\{0, \tfrac{1}{n}, \dots, \tfrac{n-1}{n},1\}$ with its usual lattice-order into a semi-primal algebra. We present the examples ordered decreasingly by the amount of subalgebras. 

First, turning a chain into a semi-primal algebra without any further impositions may be achieved as follows.
\begin{example}
The \emph{$n$-th general semi-primal chain} is given by 
$$\alg{T}_n = \big(\{0, \tfrac{1}{n}, \dots, \tfrac{n-1}{n},1\}, \wedge, \vee, 0, 1, (T_{\frac{i}{n}})_{i=0}^n\big),$$
where the unary operations $T_{\frac{i}{n}}$ are the ones from Definition \ref{defin:Ts}. For all $n\geq 1$ the algebra $\alg{T}_n$ is semi-primal (this immediately follows from Proposition \ref{SP-char-Ts}). Every subset of $T_n$ which contains the set $\{0,1\}$ defines a subalgebra of $\alg{T}_n$.
\end{example} 

Next we find examples among the \emph{\L ukasiewicz-Moisil algebras}, which were originally intended to give algebraic semantics for \L ukasiewicz finitely-valued logic. It turns out, however, that they encompass a bit more than that (see \cite{Cignoli1982}). The logic corresponding to these algebras is nowadays named after Moisil.
\begin{example}
The \emph{$n$-th \L ukasiewicz-Moisil chain} is given by 
$$\alg{M}_n = \big(\{0, \tfrac{1}{n}, \dots, \tfrac{n-1}{n},1\}, \wedge, \vee, \neg, 0, 1, (\tau_{\frac{i}{n}})_{i=1}^n\big),$$
where $\neg x = 1-x$ and the unary operations $\tau_{\frac{i}{n}}$ are the ones from Definition \ref{defin:Ts}.
For all $n\geq 1$, the algebra $\alg{M}_n$ is semi-primal. This follows from characterization (3) of Proposition \ref{SP-char-Ts} - we only have to check that $T_0$ is term-definable. To see this note that we can define $T_1(x) = \tau_1(x)$ and $T_0(x) = T_1(\neg x)$. 
\end{example}
We proceed with a classical example from many-valued logic among the finite \emph{MV-algebras} introduced by Chang (see \cite{Chang1958,Chang1959}). They give rise to the algebraic counterpart of \L ukasiewicz finite-valued logic. 
\begin{example}
The \emph{$n$-th Łukasiewicz chain} is given by 
$$\lucas_n = \big(\{0, \tfrac{1}{n}, \dots, \tfrac{n-1}{n},1\},\wedge, \vee, \oplus,\odot,\neg, 0, 1\big),$$
where $x\oplus y = \text{min}(x+y,1)$, $x\odot y = \text{max}(x+y-1, 0)$ and $\neg x = 1 - x$. 
For all $n\geq 1$, the algebra $\lucas_n$ is semi-primal. The proof of this fact can be found in \cite[Proposition 2.1]{Niederkorn2001}.
The subalgebras of $\lucas_n$ correspond to the divisors $d$ of $n$ and are of the form 
$$\lucas_d = \{ 0, \tfrac{k}{n}, \dots, \tfrac{(d-1)k}{n}, 1 \} \text{ where } n = kd.$$
\end{example} 

Other semi-primal chains are found among the \emph{Cornish algebras}, which generalize Ockham algebras (see \cite{Cornish1986, DaveyGair2017}).   

\begin{example} The \emph{$n$-th semi-primal Cornish chain} is given by
$$ \alg{CO}_n = \big(\{ 0, \tfrac{1}{n},\dots, \tfrac{n-1}{n},1\},\wedge,\vee, \neg, f, 0, 1\big), $$
where $\neg x = 1-x$, $f(0) = 0, f(1) = 1$ and $f(\frac{i}{n}) = \frac{i+1}{n}$ for $1\leq i \leq n-1$. For all $n\geq 1$, the algebra $\alg{CO}_n$ is semi-primal. The proof of this fact can be found in \cite[Example 5.15]{DaveyGair2017}. The only proper subalgebra of $\alg{CO}_n$ is $\{ 0,1 \}$.  
\end{example} 

Finally, among the \emph{Post-algebras} we find the well-known examples of chain-based algebras which are not only semi-primal, but even primal.   

\begin{example}
The \emph{$n$-th Post chain} is given by 
$$ \alg{P}_n = \big(\{ 0, \tfrac{1}{n},\dots, \tfrac{n-1}{n},1\}, \wedge, \vee, ', 0, 1\big)$$
where $1' = 0$ and $(\frac{i}{n})' = (\frac{i+1}{n})$ for $0 \leq i < n$. For all $n\geq 1$, the algebra $\alg{P}_n$ is primal (see, \emph{e.g.}, \cite[Theorem 35]{Foster1953a})
\end{example}

\subsubsection{Residuated Lattices}\label{example:reslattices}

For a general survey of residuated lattices we refer the reader to \cite{GalatosJipsen2007, JipsenKowalski2002}. We only consider bounded commutative residuated lattices here, with a particular focus on \emph{$\FLew$-algebras}. 

\begin{definition}
A \emph{(bounded commutative) residuated lattice} is an algebra $$\alg{R} = (R, \wedge, \vee, 0, 1, \odot, e, \rightarrow )$$ such that
$(R, \wedge, \vee, 0, 1)$ is a bounded lattice,
$(R, \odot, e)$ is a commutative monoid and 
the binary operation $\rightarrow$ satisfies the residuation condition 
$$x\odot y \leq z \Leftrightarrow x \leq y\rightarrow z.$$ 
We call $\alg{R}$ a \emph{$\FLew$-algebra} if, in addition, it satisfies $e = 1$.    
\end{definition}
 
Our main tool to identify semi-primal $\FLew$-algebras is \cite[Theorem 3.10]{Kowalski2004}, which implies that a $\FLew$-algebra $\alg{R}$ is quasi-primal if and only if there is some $n\geq 1$ such that
\begin{equation}\label{eq:QPFLEW}
x \vee \neg(x^n) = 1 \text{ for all } x\in R,
\end{equation}
where, as usual, we define $\neg x$ as $x\rightarrow 0$ (and $x^n$ refers to the $n$-th power with respect to $\odot$). For our purposes this theorem has the following practical consequence. 

\begin{corollary}\label{cor:QPFLEW}
Let $\alg{R}$ be a finite $\FLew$-algebra. If $\alg{R}$ does not contain any idempotent elements (that is, elements with $x \odot x = x$) other than $0$ and $1$, then $\alg{R}$ is quasi-primal. If $\alg{R}$ is based on a chain, the converse also holds.     
\end{corollary} 

\begin{proof}
Let $\alg{R}$ be a finite $\FLew$-algebra with no other idempotent elements than $0$ and $1$. Recall that, for any $a\in R$, we have 
$ \neg a = a \rightarrow 0 = \bigvee \{ b\in R \mid a \odot b \leq 0 \}.$ 
Let $a\in R{\setminus}\{ 0,1 \}$. We show that there is some $n_{a}$ such that $a^{n_a} = 0$. Since $a$ is not idempotent we have $a^2 < a$. Either $a^2 = 0$ and we are done or $a^2$ is again not idempotent. In this case we have $a^4 < a^2$ and we repeat the argument. Since $\alg{R}$ is finite, continuing this process we eventually need to find $a^{2^k} = 0$. Now $\alg{R}$ satisfies equation (\ref{eq:QPFLEW}) for $n = \bigvee \{ n_a \mid a\in R{\setminus}\{ 0,1 \} \}$, since we always have 
$$a \vee \neg (a^n) = a \vee \neg 0 = a \vee 1 = 1.$$
Thus $\alg{R}$ is quasi-primal. 

Now suppose that $\alg{R}$ is based on a chain. If $a\in R{\setminus}\{ 0,1 \}$ is idempotent, then $\neg a < a$ since for all $b \geq a$ we have $a\odot b \geq a \odot a = a$. Therefore, for all $n\geq 1$ we have $a \vee \neg (a^n) = a\vee \neg a = a \neq 1$. Thus, $\alg{R}$ does not satisfy equation (\ref{eq:QPFLEW}) and is not quasi-primal.    
\end{proof} 

\begin{remark}
The second part of the argument really requires $\alg{R}$ to be based on a chain. For example, consider the $4$-element diamond lattice $0 \leq a, b \leq 1$ with $a\wedge b = 0$ and $a\vee b = 1$. We can define a $\FLew$-algebra based on this lattice by stipulating $a^2 = a$, $b^2 = b$ and $a\odot b = 0$. Even though $a$ and $b$ are idempotent, we have $a \vee \neg a = a \vee b = 1$ and $b \vee \neg b = b \vee a = 1$. Therefore, this algebra is quasi-primal (it is, however, not semi-primal, since it has the non-trivial automorphism swapping $a$ and $b$). \hfill $\blacksquare$   
\end{remark}

In \cite{GalatosJipsen2017} Galatos and Jipsen provide a list of all finite residuated lattices of size up to $6$. Corollary \ref{cor:QPFLEW} enables us to find quasi-primal $\FLew$-algebras among them and thus, using Proposition \ref{SP-char-discriminator}, we can identify the semi-primal ones by ruling out the existence of non-trivial internal isomorphisms.  For example, there is a total of six quasi-primal $\FLew$-chains with $5$ elements ($R_{1,17}^{5,1}, R_{1,18}^{5,1} \dots R^{5,1}_{1,22}$ in \cite{GalatosJipsen2017}), five of which are semi-primal (all except $R_{1,17}^{5,1}$). Examples of semi-primal $\FLew$-algebras not based on a chain are, \emph{e.g.}, $R^{6,2}_{1,11}$ and $R^{6,3}_{1,9}$ in \cite{GalatosJipsen2017}. The algebras in question are depicted in Appendix \ref{Appendix}, where we also provide detailed proofs of these claims.   

While until now we discussed how to identify semi-primal $\FLew$-algebras, we end this subsection with two examples of semi-primal algebras based on residuated lattices where $1 \neq e$.

Specifically, we consider the \emph{bounded De Morgan monoids} $\alg{C^{01}_4}$ and $\alg{D^{01}_4}$ depicted in Figure \ref{fig:DeMorganmonoid}. 

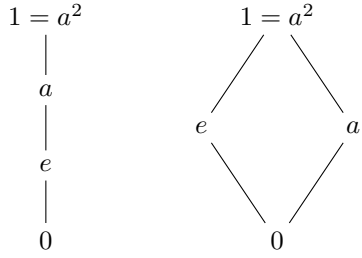
\begin{figure}[ht]
\begin{tikzpicture}
  \node (0) at (0,0) {$0$};
  \node (e) at (0,1) {$e$};
  \node (a) at (0,2) {$a$};
  \node (1) at (0,3) {$1 = a^2$};
  \draw (0) -- (e) -- (a) -- (1);
\end{tikzpicture}
\hspace{1cm}
\begin{tikzpicture}
  \node (0) at (0,0) {$0$};
  \node (1) at (0,3) {$1 = a^2$};
  \node (a) at (-1,1.5) {$e$};
  \node (b) at (1,1.5) {$a$};
  \draw (0) -- (a) -- (1) -- (b) -- (0);
\end{tikzpicture} 
 \caption{The (semi-)primal bounded De Morgan monoids $\alg{C^{01}_4}$ and $\alg{D^{01}_4}$.}
 \label{fig:DeMorganmonoid}
\end{figure}

They are bounded commutative residuated lattices with an additional involution ${\sim}$ which, in both examples, is defined by ${\sim} e = a$ and ${\sim} 0 = 1$. Our names for these algebras are inspired by \cite{Moraschini2019}, where $\alg{C_4}$ and $\alg{D_4}$ are used for the corresponding De Morgan monoids with the bounds $0$ and $1$ excluded from the signature (in \cite{Moraschini2019} it is shown that each of these two algebras generates a minimal subvariety of the variety of all De Morgan monoids).

\begin{proposition}
The algebras $\alg{C_4^{01}}$ and $\alg{D_4}^{01}$ are primal. Their reducts obtained by removing the neutral element $e$ from the signature, are semi-primal. 
\end{proposition}

\begin{proof}
Starting with $\alg{C_4^{01}}$, we directly verify that it satisfies characterization (3) of Proposition \ref{SP-char-Ts}. First we define $T_1$ and, therefore, $T_0(x) = T_1({\sim} x)$. As in \cite{DaveySchumann1991}, we do this by, for all $\ell \in \{ 0, e, a \}$, defining unary terms $u_\ell$ satisfying $u_\ell(1) = 1$ and $u_\ell(\ell) = 0$. For instance, we can define
such terms by 
$$ u_{0}(x) = x\wedge 1, \hspace{2mm} u_e(x) = {\sim}\big(({\sim}x)^2\big) \text{ and } u_a(x) = {\sim}\big({({\sim} x) \odot 1}\big).$$  
Through these terms we can clearly define $T_1(x) = u_0(x) \wedge u_{e}(x) \wedge u_a(x).$ 
Lastly, we need to define $\tau_{\ell}$ for $\ell \in \{ e, a \}$. Again, it suffices to find terms $\tau^\ast_\ell$ which satisfy
$$ \tau^\ast_\ell(x) = \begin{cases}
1 & \text{ if } x \geq \ell \\
\neq 1 & \text{ if } x \not\geq \ell,
\end{cases}$$ 
since then we get $\tau_\ell = T_1(\tau_\ell^\ast)$. Our desired terms are given by 
$$ \tau^\ast_e(x) = \big(({\sim}x)^2 \odot x\big) \vee x^2 \text{ and } \tau^\ast_a(x) = x^2.$$  

This concludes the proof for $\alg{C_4^{01}}$. The proof for $\alg{D_4^{01}}$ is completely analogous, except that we use $\tau^\ast_{e}(x) = \big(({\sim}x)^2 \odot x\big) \vee x$ instead. Thus we showed that these two algebras are semi-primal, and since they don't have any proper subalgebras they are primal. Since we never relied on the constant $e$ in the above, the last part of the statement follows. Note that in both cases, if we exclude $e$ from the signature then $\{ 0,1 \}$ becomes a proper subalgebra.    
\end{proof}

\subsubsection{Pseudo-logics}\label{Example: Pseudo-Logics} We illustrate how to generate more examples of semi-primal algebras which are based on a bounded lattice which is not necessarily a chain. The results and terminology are due to \cite{ClarkDavey1998, DaveySchumann1991}. A \emph{pseudo-logic}  
$$\alg{L} = (L, \wedge, \vee, ', 0, 1)$$
is a bounded lattice with an additional unary operation $'$ which satisfies $0'= 1$ and $1' = 0$. In \cite{DaveySchumann1991} it is shown that every subalgebra of $\alg{L}^2$ which is not the graph of an internal isomorphism is a product of subalgebras if the following two properties are satisfied:
\begin{enumerate}
\item There is no $a\in L{\setminus}\{ 0 \}$ with $a' = 1$,
\item For all $a\in L$ there exists an $n\geq 1$ with $a \wedge a^{(2n)} = 0$ (where $a^{(k)}$ denotes the $k$-fold iteration of $'$ on $a$). 
\end{enumerate}   
Using this and the characterization of Proposition \ref{SP-char-squaresubalg}, we can find more examples of semi-primal algebras. Here, we only need to assure that the above mentioned conditions are satisfied and that there are no non-trivial internal isomorphisms. For example, the three algebras depicted in Figure \ref{fig:pseudo-logics} are semi-primal (the pseudo-negation $'$ is indicated by dotted arrows).

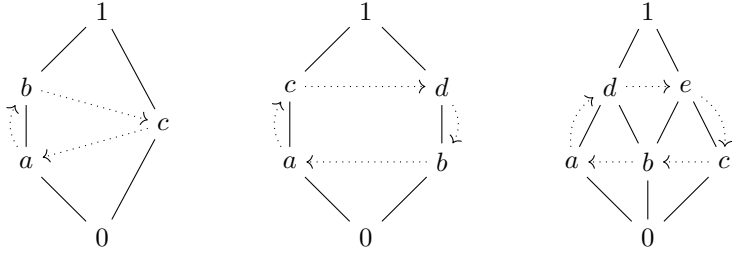
\begin{figure}[ht]
\begin{tikzpicture}
  \node (0) at (0,0) {$0$};
  \node (1) at (0,3) {$1$};
  \node (a) at (-1,1) {$a$};
  \node (b) at (-1,2) {$b$};
  \node (c) at (0.8,1.5) {$c$};
  \draw (0) -- (a) -- (b) -- (1) -- (c) -- (0);
  \path[->,font=\scriptsize]
(b) edge[dotted]   (c)
(c) edge[dotted] (a)
(a) edge[dotted, bend left] (b);
\end{tikzpicture} 
\hspace{1cm}
\begin{tikzpicture}
  \node (0) at (0,0) {$0$};
  \node (1) at (0,3) {$1$};
  \node (a) at (-1,1) {$a$};
  \node (b) at (1,1) {$b$};
  \node (c) at (-1,2) {$c$};
  \node (d) at (1,2) {$d$};
  \draw (0) -- (a) -- (c) -- (1) -- (d) -- (b) -- (0);
  \path[->,font=\scriptsize]
  (a) edge[dotted, bend left]   (c)
  (b) edge[dotted]   (a)
  (c) edge[dotted]   (d)
  (d) edge[dotted, bend left]   (b);
\end{tikzpicture}
\hspace{1cm}
\begin{tikzpicture}
  \node (0) at (0,0) {$0$};
  \node (1) at (0,3) {$1$};
  \node (a) at (-1,1) {$a$};
  \node (b) at (0,1) {$b$};
  \node (c) at (1,1) {$c$};
  \node (d) at (-0.5,2) {$d$};
  \node (e) at (0.5,2) {$e$};
  \draw (0) -- (a) -- (d) -- (1) -- (e) -- (c) -- (0);
  \draw (0) -- (b) -- (d);
  \draw (b) -- (e);
  \path[->,font=\scriptsize]
  (a) edge[dotted,bend left]   (d)
  (b) edge[dotted]   (a)
  (c) edge[dotted]   (b)
  (d) edge[dotted]   (e)
  (e) edge[dotted, bend left]   (c)  
  ;
\end{tikzpicture}
 \caption{Some semi-primal pseudo-logics (\cite{DaveySchumann1991, ClarkDavey1998}).}
 \label{fig:pseudo-logics}
\end{figure}

\subsubsection{Murskiĭ's Theorem}
While semi-primal algebras may seem rare, quite the opposite is suggested by the following. In 1975, Murskiĭ proved his surprising theorem about the proportion of semi-primal algebras of a fixed signature under increasing order. The original paper \cite{Murskii1975} is in Russian, the version we recall here is due to \cite[Section 6.2]{Bergman2011}.

\begin{theorem}{\cite{Murskii1975}}
Let $\sigma$ be an algebraic type which contains at least one operation symbol which is at least binary. Let $A_{\sigma, n}$ be the number of algebras of type $\sigma$ and size $n$ and let $SP_{\sigma, n}$ be the number of such algebras which are semi-primal. Then 
$$ \lim_{n\to\infty} \frac{SP_{\sigma,n}}{A_{\sigma,n}} = 1.$$   
\end{theorem} 
\section{Semi-primal duality}\label{sec:semi-primal duality}

One of the nice features of the variety of Boolean algebras $\BA$ is the famous \emph{Stone duality} \cite{Stone1936}. Categorically speaking, it asserts that there is a dual equivalence between $\BA$ and the category $\Stone$ of Stone spaces (that is, compact, Hausdorff and zero-dimensional topological spaces) with continuous maps:

\begin{equation*} 
  \xymatrix@C=50pt{ 
    { \Stone } \ar@/^/[rr]^\Pi & & {\ \BA }  \ar@/^/[ll]^\Sigma}  
\end{equation*} 
The functor $\Sigma$ assigns to a Boolean algebra $\alg{B}$ its collection of ultrafilters and the functor $\Pi$ assigns to a Stone space $X$ the Boolean algebra of its clopen subsets with the usual set-theoretical Boolean operations. Note that these functors can be defined on objects by 
$$ \Sigma(\alg{B}) = \BA(\alg{B},\alg{2}) \text{ and } \Pi(X) = \Stone(X,2),$$
where in the latter equation $2$ denotes the two-element discrete space. 

Stone duality has been extended to quasi-primal algebras by Keimel and Werner in \cite{KeimelWerner1974}. This duality fits the general framework of \emph{Natural Dualities}. For us, the Semi-primal Strong Duality Theorem \cite[Theorem 3.3.14]{ClarkDavey1998} is of high importance. However, we present it self-contained and in a way which particularly suits our purpose. Furthermore, we will use categorical constructions to provide a new proof of this duality. Such a proof has, to the best of our knowledge, not appeared in the literature yet. 

First we introduce the dual category of $\var{A}$ generated by a semi-primal algebra. In the following, we always consider $\mathbb{S}(\alg{L})$ as a complete lattice in its usual ordering.  

\begin{definition}\label{def:StoneL}
The category $\StoneL$ has objects $(X, \val{v})$ where $X \in \Stone$ and 
$$\val{v}\colon X \to \mathbb{S}(\alg{L})$$
assigns to every point $x\in X$ a subalgebra $\val{v}(x) \leq \alg{L}$, such that for every subalgebra $\alg{S} \leq \alg{L}$ the preimage $\val{v}^{-1}(\alg{S}{\downarrow})$ is closed.   
A morphism $m\colon (X,\val{v}) \to (Y, \val{w})$ in $\StoneL$ is a continuous map $X\to Y$ which, for all $x\in X$, satisfies 
$$ \val{w}(m(x)) \leq \val{v}(x).$$
\end{definition}  
\begin{remark}\label{dualcategoryX}
In the framework of natural dualities \cite{ClarkDavey1998}, the dual category of $\mathcal{A}$ is defined slightly differently, using Stone spaces with unary relations (i.e., subsets). Let $\var{X}$ be the category with objects $(X, \{ R_\alg{S} \mid \alg{S}\leq \alg{L}\})$, where $X \in \Stone$ and $R_\alg{S}$ is a closed subset of $X$ for each subalgebra $\alg{S} \leq \alg{L}$, satisfying
$R_\alg{L} = X$ and $R_{\alg{S_1}} \cap R_{\alg{S_2}} = R_{\alg{S}_1 \cap \alg{S_2}}$ for all $\alg{S_1},\alg{S_2} \leq \alg{L}$. 
A morphism $m\colon (X, \{ R_\alg{S}\mid \alg{S} \leq \alg{L} \}) \to (X', \{ R'_\alg{S}\mid \alg{S} \leq \alg{L} \})$ in $\var{X}$ is a continuous relation-preserving map $X\to X'$, i.e., it satisfies $x\in R_\alg{S} \Rightarrow m(x)\in R'_\alg{S}$ for all $x\in X$ and $\alg{S} \in \mathbb{S}(\alg{L})$. 

The categories $\var{X}$ and $\StoneL$ are isomorphic, as witnessed by the following mutually inverse functors $\phi$ and $\psi$. The functor $\phi : \var{X} \to \StoneL$ is given on objects by $(X, \{ R_{\alg{S}}\mid \alg{S}\leq \alg{L} \}) \mapsto (X, \val{v})$, where $$\val{v}(x) = \bigcap\{ \alg{S} \mid x\in R_{\alg{S}} \}.$$ 
The functor $\psi\colon \StoneL \to \var{X}$ is given on objects by $(X, \val{v})\mapsto (X,  \{ R_{\alg{S}} \mid \alg{S} \leq \alg{L}\})$ where 
$$ R_{\alg{S}} = \{ x\in X \mid \val{v}(x) \leq \alg{S} \}.$$
Both $\phi$ and $\psi$ map every morphism to itself.  \hfill $\blacksquare$  
\end{remark}
We now describe the two contravariant functors $\Sigma_\alg{L}$ and $\Pi_\alg{L}$ which give rise to the duality between $\var{A}$ and $\StoneL$:

\begin{equation*} 
  \xymatrix@C=50pt{ 
    \StoneL \ar@/^/[rr]^{\Pi_\alg{L}} & & {\ \var{A} }  \ar@/^/[ll]^{\Sigma_\alg{L}}}  
\end{equation*} 
On objects $\alg{A}\in \var{A}$, let the functor $\Sigma_\alg{L}$ be defined by 
$$ \Sigma_\alg{L}(\alg{A}) = \big(\var{A}(\alg{A},\alg{L}), \val{im}\big)$$
where $\val{im}$ assigns to a homomorphism $h\colon \alg{A}\to \alg{L}$ its image $\val{im}(h) = h(A) \in \mathbb{S}(\alg{L})$. A clopen subbasis for the topology on $\var{A}(\alg{A},\alg{L})$ is given by the collection of sets of the following form with $a\in A$ and $\ell\in L$: 
$$[a:\ell] = \{ h\in \var{A}(\alg{A},\alg{L}) \mid h(a) = \ell \}.$$
On morphisms $f \in \var{A}(\alg{A}_1,\alg{A}_2)$ the functor acts via composition
\begin{align*}
\Sigma_\alg{L}f\colon \var{A}(\alg{A}_2,\alg{L}) &\to \var{A}(\alg{A}_1,\alg{L}) \\
h &\mapsto h \circ f.
\end{align*}
Note that this is a morphism in $\StoneL$ since $\val{im}(h\circ f) \leq \val{im}(h)$.

Before we define the functor $\Pi_\alg{L}$, we describe the canonical way to consider $L$ as a member of $\StoneL$. Simply endow $L$ with the discrete topology and 
$$ \bm{\langle \cdot \rangle} \colon L \rightarrow \mathbb{S}(\alg{L})$$ 
assigning to an element $\ell \in L$ the subalgebra $\langle \ell \rangle \leq \alg{L}$ it generates. Now, as expected, we can define the functor $\Pi_\alg{L}$ on objects $(X, \val{v}) \in \StoneL$ by
$$ \Pi_\alg{L}(X,\val{v}) = \StoneL\big((X,\val{v}), (L,\bm{\langle\cdot\rangle})\big)$$
with the algebraic operations defined pointwise. 
This means that the carrier-set of $\Pi_\alg{L}(X,\val{v})$ is the set of continuous maps $g\colon X\to L$ which respect $\val{v}$ in the sense that, for all $x\in X$, they satisfy  
$$ g(x) \in \val{v}(x).$$
Again, on morphisms $m\colon (X,\val{v})\to(Y,\val{w})$ the functor is defined via composition
\begin{align*}
\Pi_\alg{L}m\colon \StoneL\big((Y,\val{w}), (L,\bm{\langle\cdot\rangle})\big) &\to \StoneL\big((X,\val{v}), (L,\bm{\langle\cdot\rangle})\big) \\
g &\mapsto g \circ m.
\end{align*} 
This is well-defined due to the condition on morphisms in $\StoneL$: 
$$(g\circ m)(x) = g(m(x)) \in \val{w}(m(x)) \subseteq \val{v}(x).$$ 
It is also clearly a homomorphism since the operations are defined pointwise.  
 
\begin{theorem}{\cite{KeimelWerner1974, ClarkDavey1998}}\label{SPDuality}
The functors $\Pi_\alg{L}$ and $\Sigma_\alg{L}$ are fully faithful and establish a dual equivalence between $\var{A}$ and $\StoneL$.
\end{theorem} 
The remainder of this section is dedicated to an alternative proof of this theorem. The idea is to directly prove the duality on the finite level, and then lift it to the infinite level using the following categorical constructions. 

\begin{definition}
For a finitely complete and cocomplete category $\cate{C}$, its completion under filtered colimits is denoted by $\Ind(\cate{C})$ and, dually, its completion under cofiltered limits is denoted by $\Pro(\cate{C})$. 
\end{definition} 
For example, $\Ind(\BA^\omega) \simeq \BA$ and $\Pro(\Set^\omega) \simeq \Stone$. 
More material about these completions can be found in Johnstone's book \cite[Chapter VI]{Johnstone1982} (in particular, a more rigorous definition of the $\Ind$-completion is given in VI.1.2). We only recite the following, which allows us to lift dualities between small categories (following Johnstone, dualities arising this way are called \emph{Stone type} dualities).  

\begin{lemma}{\cite[Lemma VI 3.1]{Johnstone1982}}\label{DualityIndC-ProD}
Let $\cate{C}$ and $\cate{D}$ be small categories which are dually equivalent. Then $\Ind(\cate{C})$ is dually equivalent to $\Pro(\cate{D})$. 
\end{lemma}
Our argument to prove Theorem \ref{SPDuality} now has the following outline. The role of $\cate{C}$ will be played by $\var{A}^\omega$. Since $\var{A}$ is locally finite (see, \emph{e.g.}, \cite[Lemma 1.3.2]{ClarkDavey1998}), it is well-known that $\Ind(\var{A}^\omega) \simeq \var{A}$ (see, \emph{e.g.}, {\cite[Corollary VI 2.2]{Johnstone1982}}). The role of $\cate{D}$ will be played by $\StoneL^\omega$. Since the topology doesn't matter here (because it is always discrete), we will denote this category by $\SetL^\omega$ instead. To get the finite dual equivalence, we make the following observation

\begin{proposition}\label{prop:homsareprojections}
Let $\alg{S_1},\dots,\alg{S_n}$ be subalgebras of $\alg{L}$. Then the set of homomorphisms $\mathcal{A}(\prod_{i\leq n} \alg{S_i}, \alg{L})$ consists exactly of the projections followed by inclusions $$\pr_i\colon \prod_{i\leq n} \alg{S_i} \to \alg{S_i} \hookrightarrow \alg{L}$$ in each component $i\leq n$.
\end{proposition}
\begin{proof}
Our proof is similar to that of \cite[Theorem 2.5]{ChajdaGoldstern2018}.
Let $h\colon\prod_{i\leq n} \alg{S_i}\to \alg{L}$ be a homomorphism. Since $\mathcal{A}$ is congruence distributive (Proposition \ref{SP-char-simplearithmetic}), it has the Fraser-Horn property, meaning that the congruence $\theta := \mathsf{ker}(h)$ is a product of congruences $\theta_i$ on $\alg{S_i}$. By the isomorphism theorem we find
$$ (\prod_{i\leq n} \alg{S_i})/\theta \cong \prod_{i\leq n} (\alg{S_i}/\theta_i)\cong \im(h).$$
Since $\im(h)$ is a subalgebra of $\alg{L}$ and thus simple by Proposition~\ref{SP-char-simplearithmetic}, at most one factor of $\prod_{i\leq n} (\alg{S_i}/\theta_i)$ can be non-trivial. Since $\im(h)$ contains at least two elements (that is, $0$ and $1$), precisely one factor, say $\alg{S_j}/\theta_j$, is non-trivial. Since $\alg{S_j}$ is itself semi-primal, it is simple, so $\alg{S_j}/\theta_j \cong \alg{S_j}$. So $h$ induces an internal isomorphism $\alg{S_j} \cong \im(h)$, but by Proposition \ref{SP-char-discriminator} this can only be the identity on $\alg{S_j}$, thus $h$ coincides with $\pr_j$.
\end{proof} 
\begin{corollary}\label{cor:finiteeq}
The (restrictions of the) functors $\Pi_\alg{L}$ and $\Sigma_\alg{L}$ establish a dual equivalence between the small categories ${\SetL}^\omega$ and $\var{A}^\omega$.   
\end{corollary}
\begin{proof}
Let $(X,\val{v}) \in \SetL^\omega$. Then 
$$\Sigma_\alg{L}\Pi_\alg{L}(X,\val{v}) = \Big(\mathcal{A}\big(\prod_{x\in X} \val{v}(x), \alg{L}\big),\val{im}\Big).$$ By Proposition \ref{prop:homsareprojections} this is equal to $(\{ \pr_x \mid x\in X \}, \val{im})$, which is clearly isomorphic to $(X,\val{v})$. 

On the other hand, starting with $\alg{A} \in \mathcal{A}^\omega$, we know by Proposition \ref{Afin = PS(L)} that it is a product of subalgebras $\alg{A} = \prod_{i\leq n} \alg{S_i}$. Now, again due to Proposition \ref{prop:homsareprojections}, we get $\Sigma_\alg{L}(\alg{A}) = (\{ \pr_i \mid i\leq n \}, \val{im})$, and thus 
$$\Pi_\alg{L}\Sigma_\alg{L}(\alg{A}) \cong \prod_{i\leq n} \val{im}(\pr_i) \cong \prod_{i\leq n} \alg{S_i}.$$ 

To see that $\Pi_\alg{L}$ and $\Sigma_\alg{L}$ form a dual adjunction we note that for $\alg{A} = \prod_{i\leq n}\alg{S_i} \in \var{A}^\omega$ and $(X,\val{v})\in \SetL^\omega$ we have 
$$ \var{A}^\omega\big(\Pi_\alg{L}(X,\val{v}), \alg{A}\big) \cong \prod_{i\leq n} \var{A}^\omega\big(\Pi_\alg{L}(X,\val{v}),\alg{S_i}\big) $$
and 
$$\SetL^\omega\big(\Sigma_\alg{L}(\alg{A}),(X,\val{v})\big) \cong\SetL^\omega(\coprod_{i{\leq n}}(\{ \pr_i\}, \val{im}),(X,\val{v})) \cong \prod_{i{\leq n}} \SetL^\omega\big((\{ \pr_i\}, \val{im}),(X,\val{v})\big)$$
where the coproduct in $\SetL^\omega$ is the obvious disjoint union. So we only need to show that 
$$ \var{A}^\omega\big(\Pi_\alg{L}(X,\val{v}),\alg{S_i}\big) \cong \SetL^\omega\big((\{ \pr_i\}, \val{im}), (X,\val{v})\big).$$
But this is obvious since the elements of the left-hand side are exactly the projections with image contained in $\alg{S_i}$, which are in bijective correspondence with the points of $X$ with $\val{v}(x)\leq \alg{S_i}$, that is, with elements of the right-hand side.    
\end{proof}
In order to successfully apply Lemma \ref{DualityIndC-ProD}, it remains to show the following.

\begin{theorem}\label{thm:ProSetL=StoneL}
$\Pro(\SetL^\omega)$ is categorically equivalent to $\StoneL$.
\end{theorem}  
\begin{proof}
First we show that the category $\StoneL$ is complete. For an index set $I$ (which we often omit), we claim that the product is computed as 
$$ \prod_{i\in I} (X_i, \val{v_i}) = (\prod_{i\in I} X_i, \bigvee \val{v_i}),$$
where $\bigvee \val{v_i} (p) = \bigvee (\val{v_i}(p_i))$ for all $p\in \prod X_i$. It follows from
$$(\bigvee \val{v_i})^{-1}(\alg{S}{\downarrow}) = \prod \val{v_i}^{-1}(\alg{S}{\downarrow})$$ that this defines a member of $\StoneL$.   
Note that the projections are morphisms in $\StoneL$ since 
$$ \val{v_i}(\pi_i(p)) = \val{v_i}(p_i) \leq \bigvee_{j\in I} \val{v_j}(p_j) = (\bigvee \val{v_j})(p).$$
If $(\gamma_i\colon (Y,\val{w})\to (X_i,\val{v_i}) \mid i\in I)$ is another cone, there is a unique continuous map $f\colon Y\to \prod X_i$ with $\pi_i \circ f = \gamma_i$. This map is a morphism in $\StoneL$ since 
$$ (\bigvee \val{v_i})(f(y)) = \bigvee \val{v_i}\big(\pi_i(f(y))\big) = \bigvee \val{v_i}\big(\gamma_i(f(y))\big) \leq \val{w}(y),$$
where the last inequality follows from $\val{v_i}(\gamma_i)(y) \leq \val{w}(y)$ which is true since the $\gamma_i$ are morphisms in $\StoneL$. 
The equalizer of $f,g\colon (X,\val{v})\to (Y,\val{w})$ is simply given by $(Eq, \val{v}{\mid}_{Eq})$ where $Eq \subseteq X$ is the corresponding equalizer in $\Stone$. It follows that $\StoneL$ has all limits. In particular, $\StoneL$ has all cofiltered limits, so the natural inclusion functor $\iota\colon\SetL^\omega \hookrightarrow \StoneL$ has a unique cofinitary (that is, cofiltered limit preserving) extension 
$$\hat{\iota} \colon \Pro(\SetL^\omega) \hookrightarrow \StoneL.$$
Since $\iota$ is fully faithful, to conclude that the functor $\hat{\iota}$ is fully faithful as well it suffices to show that $\iota$ maps all objects to finitely copresentable objects in $\StoneL$ (this is due to the analogue of \cite[Theorem VI.1.8]{Johnstone1982} for the $\Pro$-completion). So we need to show that any $(C,\val{w}) \in \StoneL$ where $C$ is a finite discrete space is finitely copresentable. In other words, we need to show that, whenever $(X, \val{v}) \cong \lim_{i\in I} (X_i, \val{v_i})$ is a cofiltered limit of a diagram $(f_{ij}\colon (X_j,\val{v_j})\to (X_i, \val{v_i}) \mid i\leq j)$ in $\StoneL$ with limit morphisms $p_i\colon (X,\val{v})\to (X_i,\val{v_i})$, any morphism $f\colon (X,\val{v})\to (C,\val{w})$ factors essentially uniquely through one of the $p_i$. For this we can employ an argument similar to the one in the proof of \cite[Lemma 1.1.16(b)]{RibesZalesskii2010}. On the underlying level of $\Stone$, where finite discrete spaces are finitely copresentable, the continuous map $f$ factors essentially uniquely through some $p_i$, say via the continuous map $g_i \colon X_i \to C$. However, $g_i$ is not necessarily a morphism in $\StoneL$. Consider $J = \{ j \geq i\}$ and for each $j\in J$ define $g_j = f_{ij} \circ g_i$. Define the continuous maps $\mu \colon X\to \mathbb{S}(\alg{L})^2$ and $\mu_j\colon X_j \to \mathbb{S}(\alg{L})^2$ for all $j\in J$ by
$$\mu(x) = \big(\val{w}(f(x)), \val{v}(x)\big) \text{ and } \mu_j(x) = \big(\val{w}(g_j(x)), \val{v_j}(x)\big).$$   
Since $\mu(X) = \lim_{j\in J} \mu_j(X_j) = \bigcap_{j \geq i} \mu_j(X_j)$ is contained in the finite set $\mathbb{S}(\alg{L})^2$ and $J$ is directed, there is some $k\in J$ such that 
$$ \mu(X) = \mu_k(X_k).$$
But now, since $f$ is a morphism in $\StoneL$, we have that $\mu(X) \subseteq \{ (\alg{S},\alg{T}) \mid \alg{S} \leq \alg{T} \}$, and thus the same holds for $\mu_k(X_k)$. Thus $g_k$ is a morphism in $\StoneL$ which has the desired properties. 
 
To finish the proof we show that $\hat{\iota}$ is essentially surjective, in other words, we show that every element $(X,\val{v})$ of $\StoneL$ is isomorphic to a cofiltered limit of elements of $\SetL^\omega$. We do this in a manner similar to \cite[Theorem 1.1.12]{RibesZalesskii2010}. Let $\mathcal{R}$ consist of all finite partitions of $X$ into clopen sets. Together with the order $R \leq R'$ if and only if $R'$ refines $R$ this forms a codirected set and in \cite[Theorem 1.1.12]{RibesZalesskii2010} it is shown that $X \cong \lim_{R\in \mathcal{R}} R$. We now turn every $R\in \mathcal{R}$ into a member of $\SetL^\omega$ by endowing it with an appropriate $v_R\colon R \to \mathbb{S}(\alg{L})$ and show that $(X,v) = \lim_{R\in \mathcal{R}}(R,v_R)$. For $R\in \mathcal{R}$, say $R = \{ \Omega_1, \dots, \Omega_k \}$, we define 
$$v_R^{-1}(\alg{S}{\downarrow}) = \{ \Omega_i \mid \Omega_i \cap \val{v}^{-1}(\alg{S}{\downarrow})  \neq \varnothing\}.$$ 
The map $p_R\colon X \to R$ defined by $p_R(x) = \Omega_i \Leftrightarrow x\in \Omega_i$ is a morphism in $\StoneL$ since $\val{v}(x) = \alg{S}$ and $x\in \Omega_i$ implies $v_R(p_R(x)) \in v_R^{-1}(\alg{S}{\downarrow})$. Is is easy to see that this defines a cone over the diagram $(R,v_R)_{R\in \mathcal{R}}$, so there is a unique $f\colon (X,\val{v}) \to \lim_{R\in\mathcal{R}}(R,v_R)$ in $\StoneL$. As in $\Stone$, the map $f$ is a homeomorphism. To complete the proof it suffices to show that $f^{-1}$ is a morhpism in $\StoneL$ as well. Say $\lim_{R\in\mathcal{R}}(R, v_R) = (Y, \val{w})$ and let $\pi_R \colon (Y,\val{w})\to (R,v_R)$ denote the limit morphisms.  Assuming $\val{w}(y) = \alg{S}$ we want to show $f^{-1}(y) \in \val{v}^{-1}(\alg{S}{\downarrow})$. Let $\Omega \subseteq X$ be an arbitrary clopen set containing $f^{-1}(y)$. Then $R = \{ \Omega, X{\setminus}\Omega \} \in \mathcal{R}$ and 
$$\Omega = p_R(f^{-1}(y)) = \pi_R(y) \in v_R^{-1}(\alg{S}{\downarrow}).$$
By definition this means that $\Omega \cap \val{v}^{-1}(\alg{S}{\downarrow}) \neq \varnothing$. Since this holds for every $\Omega$ containing $f^{-1}(y)$, this implies that $f^{-1}(y)$ is in the closure $\overline{\val{v}^{-1}(\alg{S}{\downarrow})}$. However, this closure coincides with $\val{v}^{-1}(\alg{S}{\downarrow})$, since by definition of $\StoneL$ this is a closed set already.        
\end{proof} 

As discussed before, this yields our alternative proof of Theorem \ref{SPDuality}. In Section \ref{sec:canext} we investigate the other dual equivalence which can be obtained from the finite dual equivalence of Corollary \ref{cor:finiteeq}. More specifically, there we describe $\Ind(\SetL^\omega)$ and its dual, the category of profinite algebras $\Pro(\var{A}^\omega)$. This is the `semi-primal version' of the duality between $\Ind(\Set^\omega) \simeq \Set$ and $\Pro(\BA^\omega) \simeq \CABA$. 

Before that, in the following section we investigate the relationship between $\StoneL$ and $\Stone$ and, more interestingly, between $\var{A}$ and $\BA$.        

\section{A chain of adjuntions}\label{sec:adcuntions} In this section we explore the relationship between Stone duality and the semi-primal duality discussed in the previous section. We start with the connection between $\StoneL$ and $\Stone$, which will be expressed in terms of a chain of four adjoint functors (similar to one in \cite{Walker2004}). Then we look at the duals of these functors and give them purely algebraic descriptions to gain insight into the structure of $\var{A}$ relative to that of $\BA$. The entire situation is summarized in Figure \ref{fig:adjunctions}, which we will have fully described at the end of this section (note that left-adjoints on the topological side correspond to right-adjoints on the algebraic side and vice-versa, since the functors $\Pi_\alg{L},\Sigma_\alg{L}$ and $\Sigma, \Pi$ which establish the two dualities are contravariant).  

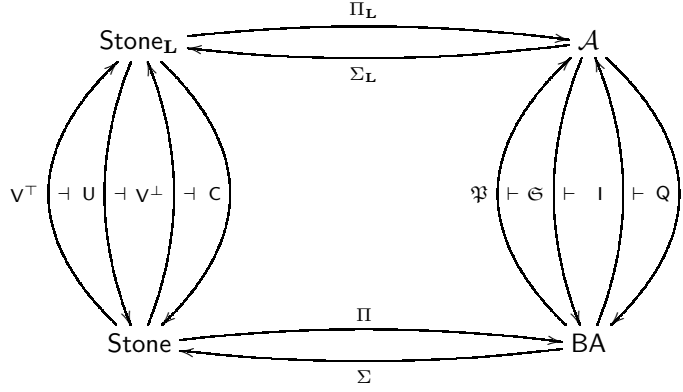
\begin{figure}[ht]
$$
\xymatrix@C=140pt@R=100pt{ 
\StoneL \ar@/^/[r]^{\Pi_\alg{L}}  
\ar@{<-}@/_34pt/[d]^\dashv_{\Vtop}
\ar@/_13pt/[d]_{\U}^{\dashv}
\ar@{<-}@/^13pt/[d]_{\Vbot}^{\dashv}
\ar@/^34pt/[d]_{\C}
& 
\var{A} \ar@/^/[l]^{\Sigma_\alg{L}}
\ar@{<-}@/_34pt/[d]^{\vdash}_{\M}
\ar@/_13pt/[d]_{\B}^{\vdash}
\ar@{<-}@/^13pt/[d]_{\I\phantom{a}}^{\vdash}
\ar@/^34pt/[d]_{\Q}  
\\
\Stone \ar@/^/[r]^{\Pi}  
& 
\BA \ar@/^/[l]^{\Sigma}  
}
$$ 
\caption{The chain of adjunctions on the topological and the algebraic side.}
\label{fig:adjunctions}
\end{figure}    

\subsection{Four functors on the topological side}\label{subs:fourfunctors}

Let $\U\colon \StoneL \to \Stone$ be the obvious forgetful functor. This functor has a left-adjoint and a right-adjoint $\Vtop \dashv \U \dashv \Vbot$. The two functors $\Vtop, \Vbot\colon\StoneL\rightarrow\Stone$ are given on objects by 
\begin{align*}
& V^\top(X) = (X, \val{v^\top}) \text{ where } \forall x\in X: \val{v^\top}(x) = \alg{L}, \\ 
& V^\bot(X) = (X, \val{v^\bot}) \text{ where } \forall x\in X: \val{v^\bot} (x) = \langle 0,1 \rangle
\end{align*} 
and both assign every morphism to itself. Here $\langle 0,1 \rangle$ is the subalgebra generated by $\{ 0,1 \}$, that is, the (unique) smallest subalgebra of $\alg{L}$.

To see $ V^\top \dashv \U$ note that by definition we have 
$$m \in \StoneL\big((X,\val{v^\top}), (Y,\val{w})\big) \Leftrightarrow m\in\Stone(X,Y) \wedge \forall x\in X: \val{w}(m(x)) \leq \val{v^\top} (x),$$
and $\val{w}(m(x)) \leq \val{v^\top}(x) = \alg{L}$ is trivially satisfied for every $m\in\Stone(X,Y)$. 

Similarly we see $\U \dashv \Vbot$, since every $m\in\Stone(X,Y)$ automatically satisfies $\val{v^\bot}(m(x)) \leq \val{w}(x)$ and, therefore, $m\in\StoneL\big((X,\val{w}),(Y,\val{v^\bot})\big)$. 

The functor $\Vbot$ also has a right-adjoint $\C: \StoneL \rightarrow \Stone $ defined by 
$$ \C(X,\val{v}) = \{ x\in X \mid \val{v}(x) = \langle 0,1 \rangle \}$$
on objects. On morphisms $m\colon (X,\val{v}) \to (Y,\val{w})$ it acts via restriction $m \mapsto m{\mid}_{\C(X,v)}$, which is well-defined since $m\in \StoneL\big((X,\val{v}), (Y,\val{w})\big)$ and $x\in \C(X,\val{v})$ means 
$$ \val{w}(m(x)) \leq \val{v}(x) = \langle 0,1 \rangle $$
which is equivalent to $m(x) \in \C(W,\val{w}).$
Again $\Vbot \dashv \C$ is easy to see since
$$m\in \StoneL\big((X,\val{v^\bot}), (Y,\val{w})\big) \Leftrightarrow \forall x: \val{w}(m(x)) \leq \langle 0,1 \rangle \Leftrightarrow m\in\Stone\big(X,\C (Y,\val{w})\big).$$

The functor $\Vtop$ preserves almost all limits, however, there is one important exception. The terminal object (that is, the limit of the empty diagram) in $\StoneL$ is given by $(\{ \ast \}, \val{v}^\bot)$, implying that $\Vtop$ does not preserve terminal objects. Therefore, contrary to a claim made in \cite{Walker2004}, no further left-adjoint of $\Vtop$ exists.  

It is obvious that both the unit $\id_{\Stone} \Rightarrow \U \circ \Vtop$ of the adjunction $\Vtop \dashv \U$ and the counit $\U \circ \Vbot \Rightarrow \id_{\Stone}$ of the adjunction $\U \dashv \Vbot$ are natural isomorphisms. We hold on to this fact, which will also be interesting on the algebraic side.   
\begin{proposition}\label{prop:Stonereflectivesubcategory} The category $\Stone$ is categorically equivalent to
\begin{enumerate}[(i)] 
\item a coreflective subcategory of $\StoneL$, witnessed by the fully faithful functor $\Vtop$. 
\item a reflective and coreflective subcategory of $\StoneL$, witnessed by the fully faithful functor $\Vbot$. 
\end{enumerate}
\end{proposition}
The functors described in this subsection can be carried through the dualities, resulting in a a corresponding chain of adjunctions between $\var{A}$ and $\BA$. For example, the dual of $\U$ is given by $\Pi\U \Sigma_\alg{L}\colon \var{A}\to \BA$. In the next subsection we show that this functor can be understood algebraically as the Boolean skeleton. Throughout the subsections that follow, we will give similar algebraic descriptions for all of these functors between $\var{A}$ and $\BA$ in Figure \ref{fig:adjunctions}.   

\subsection{The Boolean skeleton functor} In the theory of $\cate{MV}_n$-algebras (that is, the case where $\alg{L} = \lucas_n$), the Boolean skeleton is a well-known and useful tool (see, for example, \cite{Cignoli2000}). An appropriate generalization of this concept to our setting was made by Maruyama in \cite{Maruyama2012} (where it is called the Boolean core).

Due to Proposition \ref{SP-char-Ts} and \cite[Lemma 3.11]{Maruyama2012}, the following definition is justified.    
\begin{definition}
Let $\alg{A}\in \var{A}.$ The \emph{Boolean skeleton} of $\alg{A}$ is the Boolean algebra 
$ \B(\alg{A}) = (\B(A), \wedge, \vee, T_0, 0, 1) $
on the carrier set 
$$ \B(A) = \{ a\in A \mid T_1(a) = a\},$$
where the lattice operations $\wedge$ and $\vee$ are inherited from $\alg{A}$ and the unary operations $T_0$ and $T_1$ correspond to the ones from Definition \ref{defin:Ts} (which by Proposition \ref{SP-char-Ts} are term-definable in $\alg{L}$), interpreted in $\alg{A}$.  
\end{definition}

For example, for each $\alg{A} \in \var{A}, a\in \alg{A}$ and $\ell \in \alg{L}$ we have $T_\ell(a) \in \B(\alg{A})$. This holds since the equation $T_1(T_\ell (x)) \approx T_\ell(x)$ holds in $\alg{L}$, and therefore also in $\alg{A}$. 

\begin{remark}
For $\alg{A} \in \var{A}$, suppose that $A' \subseteq A$ is a subset such that $(A', \wedge, \vee, T_0, 0, 1)$ forms a Boolean algebra. Then, for all $a' \in A'$, we have $T_1(a') = T_1(T_0(T_0(a'))) = T_0(T_0(a')) = a'$ and thus $a' \in \B (A)$ (the second equation always holds since $\var{A} \models T_1(T_0(x)) \approx T_0(x)$, which is easily checked in $\alg{L}$). Therefore, $\B(A)$ is the largest such subset. \hfill $\blacksquare$ 
\end{remark}

To extend the construction of the Boolean skeleton to a functor $\B\colon \var{A}\to \BA$, on homomorphisms $f\in \var{A}(\alg{A}_1,\alg{A}_2)$ we define $\B f$ to be the restriction $f{\mid}_{\B(\alg{A}_1)}$. This is well-defined since 
$$ a\in \B(\alg{A}_1) \Leftrightarrow T_1(a) = a \Rightarrow T_1(f(a)) = f(T_1(a)) = f(a) \Leftrightarrow f(a) \in \B(\alg{A}_2).$$
The following is arguably the most important property of the Boolean skeleton.

\begin{proposition}\label{prop:HomeoBooleanSkeleton}
For all $\alg{A}\in\var{A}$, there is a homeomorphism between $\U\Sigma_\alg{L}(\alg{A}) = \var{A}(\alg{A},\alg{L})$ and $\Sigma\B(\alg{A}) = \BA(\B(\alg{A}), \alg{2})$ given by $h\mapsto h{\mid}_{\B(\alg{A})}$.
\end{proposition}  
\begin{proof}
First we show that the map is a bijection. For injectivity, suppose that $g$ and $h$ satisfy $g{\mid}_{\B(A)} = h{\mid}_{\B(A)}$. Take an arbitrary element $a\in \alg{A}$ and let 
$g(a) = \ell$. Using that $T_\ell(a) \in \B(\alg{A})$ we get 
$$1 = T_\ell (g(a)) = g(T_\ell(a)) = h(T_\ell(a)) = T_\ell(h(a)),$$
which implies $h(a) = \ell$ and, since $a$ was arbitrary, that $g = h$. 
For surjectivity, let $h \in \BA(\B(\alg{A}),\alg{2})$ be arbitrary. Due to \cite[Lemma 3.12]{Maruyama2012} the following yields a well-defined homomorphism $\bar{h}\in\var{A}(\alg{A},\alg{L})$:
$$ \bar{h}(a) = \ell \Leftrightarrow h(T_\ell(a)) = 1.$$
Since for $a\in\B(\alg{A})$ we have
\begin{align*}
h(T_1(a)) = 1 &\Leftrightarrow h(a) = 1 \text{ and } \\
h(T_0(a)) = 1 &\Leftrightarrow T_0(h(a)) = 1 \Leftrightarrow h(a) = 0, 
\end{align*}     
we conclude that $\bar{h}{\mid}_{\B(\alg{A})} = h$.

We now have a bijection between two Stone spaces, so it only remains to show that it is continuous. But this is easy to see since the preimage of an open subbasis element $[a:i] \subseteq \BA(\B(\alg{A}), \alg{2})$ is the open subbasis element $[a:i]\subseteq \var{A}(\alg{A},\alg{L})$. 
\end{proof} 
\begin{corollary}
There is a natural isomorphism between the functor $\B$ and the dual $\Pi\U \Sigma_\alg{L}$ of the forgetful functor $\U$.
\end{corollary}
\begin{proof}
By Proposition \ref{prop:HomeoBooleanSkeleton}, for every $\alg{A} \in \var{A}$, setting  
\begin{align*}
\phi_\alg{A}\colon \U \Sigma_\alg{L}(\alg{A}) & \to \Sigma\B (\alg{A}) \\
h &\mapsto h{\mid}_{\B(\alg{A})}
\end{align*}
defines a natural isomorphism $\phi\colon \U \Sigma_\alg{L} \Rightarrow \Sigma \B$ (naturality is easy to check using the definitions of $\Sigma, \Sigma_\alg{L}$ and $\B$ on morphisms). Applying $\Pi$ and using the fact that $\Pi\Sigma$ is naturally isomorphic to $ \id_\BA$, we get the natural isomorphism $\Pi\phi\colon \B \Rightarrow \Pi\U \Sigma_\alg{L}.$  
\end{proof}

In the next subsection we explain the right-adjoint of the Boolean skeleton functor.  
\subsection{The Boolean power functor}\label{subs:Boolpower} 
In this subsection we give an algebraic description of a functor naturally isomorphic to the dual $\Pi_\alg{L}\Vtop \Sigma$ of the functor $\Vtop$. This functor, which we call $\M$, turns out to be an instance of the the well-known \emph{Boolean power}  (or \emph{Boolean extension}), which was introduced for arbitrary finite algebras in Foster's first paper on primal algebras \cite{Foster1953a}. Boolean powers are special instances of Boolean products (see, \emph{e.g.}, \cite[Chapter IV]{BurrisSankappanavar1981}), but for our purposes it is more convenient to work with the following equivalent definition found in \cite{Burris1975}.   

\begin{definition}\label{defin:Boolpowers}
Given a Boolean algebra $\alg{B} \in \BA$ and a finite algebra $\alg{M}$, the \emph{Boolean power} $\alg{M}[\alg{B}]$ is defined on the carrier set 
$$ M[B] \subseteq B^M $$
consisting of all maps $\xi \colon M \to B$ which satisfy 
\begin{enumerate}
\item If $\ell$ and $\ell'$ are distinct elements of $M$, then $\xi(\ell) \wedge \xi(\ell') = 0$,
\item $\bigvee \{ \xi(\ell) \mid \ell\in M \} = 1$.  
\end{enumerate}
If $o^{\alg{L}}\colon M^k \to M$ is a $k$-ary operation of $\alg{M}$, we define a corresponding operation $o^{\alg{M}[\alg{B}]}\colon M[B] \to M[B]$ by 
$$ o^{\alg{M}[\alg{B}]}(\xi_1, \dots, \xi_k) (\ell) = \bigvee_{o^{\alg{M}}(\ell_1, \dots, \ell_k) = \ell} ( \xi_1(\ell_1) \wedge \dots \wedge \xi_k(\ell_k)).$$
The resulting algebra $\alg{M}[\alg{B}] = (M[B], \{ o^{\alg{M}[\alg{B}]} \mid o \text{ in the signature of } \alg{M} \})$ is a member of the variety $\variety{\alg{M}}$ generated by $\alg{M}$ (since it satisfies the same equations as $\alg{M}$).   
\end{definition}

There is a straightforward way to extend this construction to a functor.  

\begin{definition}\label{def:Boolpowerfunctor}
Given a finite algebra $\alg{M}$, we define the functor $\M_\alg{M}\colon \BA \to \variety{\alg{M}}$ as follows. On objects $\alg{B} \in \BA$ we define 
$$ \M_\alg{M}(\alg{B}) = \alg{M}[\alg{B}].$$
For a Boolean homomorphism $\varphi\colon \alg{B} \to \alg{B}'$, the homomorphism $\M_\alg{M}\varphi\colon \alg{M}[\alg{B}]\to \alg{M}[\alg{B}']$ is defined via composition $\xi \mapsto \varphi \circ \xi$ (this is a homomorphism because operations in $\alg{M}[\alg{B}]$ are defined by Boolean expressions, which commute with $\varphi$).
\end{definition}
        
In particular, we will use the shorthand notation $\M$ for $\M_{\alg{L}}$. In the remainder of this subsection we aim to show that $\M$ is indeed the right-adjoint of the Boolean skeleton functor $\B$. For this, we need the following well-known properties of the Boolean power. 
\begin{lemma}{\cite[Proposition 2.1]{Burris1975}}\label{lem:propertiesofM} The functor $\M_\alg{M}$ has the following properties:
\begin{enumerate}[(i)]
\item $\M_\alg{M}(\alg{2})\cong\alg{M}$,
\item $\M_\alg{M}$ preserves products. 
\end{enumerate}
In particular, $\M_\alg{M}(\alg{2}^\kappa) \cong \alg{M}^\kappa$ holds for all index sets $\kappa$. 
\end{lemma}

In the next proposition we describe the interplay between the functors $\B$ and $\M$. Again, the terms $T_\ell$ from Proposition \ref{SP-char-Ts} play an important role.

\begin{proposition}\label{prop:MandB}
For every $\alg{A} \in \var{A}$ there is an embedding $\mathcal{T}_{(\cdot)}\colon \alg{A} \hookrightarrow \M(\B(\alg{A}))$ given by $a \mapsto \mathcal{T}_a$ where 
$$\mathcal{T}_a(\ell) = T_\ell(a).$$
The restriction to $\B(\alg{A})$ yields an isomorphism $\B(\alg{A}) \cong \B\big(\M(\B(\alg{A}))\big)$. 
\end{proposition}

\begin{proof}
The map is well-defined, that is, $\mathcal{T}_a$ is in $\M(\B(\alg{A}))$, since the equations $T_\ell(x) \wedge T_{\ell'}(x) \approx 0$ (for distinct $\ell, \ell'$) and $\bigvee \{ T_\ell(x) \mid \ell \in L\} \approx 1$ hold in $\alg{L}$.

We now fix an embedding $\alg{A} \hookrightarrow \alg{L}^I$. It is easy to see that $\mathcal{T}_{(\cdot)}$ is injective since, for distinct $a, a' \in \alg{A}$, there is some component $i\in I$ with $a(i) = \ell \neq a'(i)$, thus $\mathcal{T}_a(\ell) \neq \mathcal{T}_{a'}(\ell)$. To conclude that $\mathcal{T}_{(\cdot)}$ is an embedding we need to show that it is a homomorphism, that is we want to show that for any $k$-ary operation $o\colon L^k \to L$ of $\alg{L}$ we have 
$$ \mathcal{T}_{o^\alg{A}(a_1, \dots, a_k)} = o^{\alg{L}[\mathfrak{S}(\alg{A})]}(\mathcal{T}_{a_1}, \dots \mathcal{T}_{a_k}).$$
By definition the $i$-th component of the left-hand side is given by 
$$\mathcal{T}_{o^{A}(a_1, \dots, a_k)}(\ell)(i) = T_\ell \big(o^\alg{L}(a_1(i), \dots, a_k(i))\big) = \begin{cases}
1 & \text{ if } o^\alg{L}(a_1(i),\dots, a_k(i)) = \ell \\
0 & \text{ otherwise.}
\end{cases}$$
The right-hand side is given by
$$ o^{\alg{L}[\mathfrak{S}(\alg{A})]}(\mathcal{T}_{a_1}, \dots \mathcal{T}_{a_k})(\ell) = \bigvee_{o^{\alg{L}}(\ell_1, \dots, \ell_k) = \ell} ( \mathcal{T}_{a_1}(\ell_1) \wedge \dots \wedge \mathcal{T}_{a_k}(\ell_k) ). $$        
In its $i$-th component this again corresponds to
$$\bigvee_{o^{\alg{L}}(\ell_1, \dots, \ell_k) = \ell} \big(T_{\ell_1}(a_1(i)) \wedge \dots \wedge T_{\ell_k}(a_k(i))\big) = \begin{cases}
1 & \text{ if } o^\alg{L}(a_1(i),\dots, a_k(i)) = \ell \\
0 & \text{ otherwise.}
\end{cases}$$
Thus $\mathcal{T}_{(\cdot)}$ is an embedding, which concludes the proof of the first statement.

For the second statement, note that, since $\B$ preserves injectivity of homomorphisms, it suffices to show that the restriction of $\mathcal{T}_{(\cdot)}$ to $\B(\alg{A})$ is a surjection onto $\B\big(\M(\B(\alg{A}))\big)$. So consider an element $\xi \in \B\big(\M(\B(\alg{A}))\big)$, that is $\xi \in \M(\B(\alg{A}))$ and $T_1^{\alg{L}[\alg{\B(\alg{A})}]}(\xi) = \xi.$ The latter by definition means  
$$ T_1^{\alg{L}[\alg{\B(\alg{A})}]}(\xi)(1) = \xi(1),$$
$$ T_1^{\alg{L}[\alg{\B(\alg{A})}]}(\xi)(0) = \bigvee\{ \xi(\ell) \mid \ell \in L, \ell \neq 1 \} = \xi(0) \text{ and } $$
$$ T_1^{\alg{L}[\alg{\B(\alg{A})}]}(\xi)(\ell) = \bigvee\varnothing = 0 = \xi(\ell) \text{ for all } \ell\in L{\setminus}\{ 0,1 \}.$$
We claim that $\xi = \mathcal{T}_{\xi(1)}.$ Indeed, we know that $\xi(1) \in \B(\alg{A})$ so $\xi(1) = T_1(\xi(1))$. 
Furthermore, in the component $i \in I$, we have $\xi(0)(i) = 1$ if and only if $\xi(1)(i) = 0$, so $T_0(\xi(1)) = T_1(\xi(0)) = \xi(0)$ since $\xi(0)\in \B(\alg{A})$. Finally, for $\ell \not\in \{ 0,1 \}$ we have $T_\ell(\xi(1)) = 0$ since for all $i\in I$ we have $\xi(1)(i) \in \{ 0,1 \}$. This concludes the proof.   
\end{proof}

Since $\B$ is dual to the essentially surjective functor $\U$, we know that every $\alg{B}\in \BA$ is isomorphic to $\B(\alg{A})$ for some $\alg{A} \in \var{A}$. Therefore, the following is a direct consequence of the second part of Proposition \ref{prop:MandB}. 

\begin{corollary}\label{cor:BisoBMB}
Every Boolean algebra $\alg{B} \in \BA$ is isomorphic to $\B(\M(\alg{B}))$.  
\end{corollary}

Another immediate consequence of Proposition \ref{prop:MandB} is the following. 
 
\begin{corollary}\label{cor:embedAinMBA}
For every Boolean algebra $\alg{B} \in \BA$, the algebra $\M(\alg{B})$ is the largest algebra in $\var{A}$ which has $\alg{B}$ as Boolean skeleton. That is, for every algebra $\alg{A} \in \var{A}$ with $\B(\alg{A}) \cong \alg{B}$ there exists an embedding $\alg{A} \hookrightarrow \M(\alg{B})$.   
\end{corollary}

We now have everything at hand to prove the main theorem of this subsection.

\begin{theorem}\label{thm:MadjointB}
$\M$ is naturally isomorphic to the dual of $\Vtop$ and, therefore,  
$$ \B \dashv \M.$$
\end{theorem}
\begin{proof}
First we prove the statement on the finite level. In other words, we want to show that, in $\StoneL$,
$$ \Sigma_\alg{L}\M(\alg{B}) \cong \Vtop\Sigma(\alg{B})$$ 
holds for every finite Boolean algebra $\alg{B}$. More explicitly, after spelling out the definition of the functors involved we want to show
\begin{equation}\label{eq: isoStoneL}
 \big(\var{A}(\M(\alg{B}),\alg{L}), \val{im}\big) \cong \big(\BA(\alg{B},\alg{2}),\val{v}^\top\big)
\end{equation}
for every finite Boolean algebra $\alg{B}$.
First, since $\alg{B}$ is finite there is some positive integer $k$ such that $\alg{B} \cong \alg{2}^k$. We combine the following isomorphisms in $\Stone$. Due to Proposition \ref{prop:HomeoBooleanSkeleton} we know
$$ \var{A}\big(\M(\alg{B}),\alg{L}\big) \cong \BA\big(\B(\M(\alg{B})),\alg{2}\big),$$
And due to Corollary \ref{cor:BisoBMB} we know 
$$\B(\M(\alg{B})) \cong \alg{B}.$$  
Putting these together, we get  
$$ \var{A}(\M(\alg{B}),\alg{L}) \cong \BA(\alg{B},\alg{2}).$$
In fact, this even yields an isomorphism in $\StoneL$ as desired in (\ref{eq: isoStoneL}), because 
$$ \big(\var{A}(\M(\alg{B}),\alg{L}), \val{im}\big) \cong \big(\var{A}(\alg{L}^k, \alg{L}), \val{im}\big) \cong \big(\var{A}(\alg{L}^k,\alg{L}),\val{v}^\top\big)$$
where the last equation holds due to Proposition \ref{prop:homsareprojections}.

So we know that the restriction of $\M$ to the category of finite Boolean algebras $\M^\omega\colon \BA^\omega \to \var{A}$ is dual to the restriction $(\Vtop)^\omega$ of $\Vtop$ to the category $\SetL^\omega$. There is a unique (up to natural iso) finitary (i.e., filtered colimit preserving) extension of $\M^\omega$ to $\Ind(\BA^\omega) \simeq \BA$, and this extension is naturally isomorphic to the dual of $\Vtop$ (since $\Vtop$ preserves all limits except for the terminal object, it is the unique cofinitary extension of $(\Vtop)^\omega$). To show that $\M$ coincides with this unique extension (up to natural isomorphism), it suffices to show that $\M$ is finitary as well. Since $\M$ preserves monomorphisms (it is easy to see by definition that if $\varphi\in\BA(\alg{B}_1,\alg{B}_2)$ is injective, then $\M\varphi$ is injective as well), we can apply \cite[Theorem 3.4]{Adamek2019}, which states that $\M$ is finitary if and only if the following holds. \\
\emph{For every Boolean algebra $\alg{B} \in \BA$, for every finite subalgebra $\alg{A} \hookrightarrow \M(\alg{B})$ the inclusion factors through the image of the inclusion of some finite subalgebra $\alg{B}' \hookrightarrow \alg{B}$ under $\M$.} 

To see this write $\alg{A} \cong \prod_{i\leq n}\alg{S_i}$ as product of finite subalgebras of $\alg{L}$. Then, by Corollary \ref{cor:BisoBMB}, we know that $\B(\alg{A}) \cong \alg{2}^n$ embeds into $\alg{B}$. Now by Lemma \ref{lem:propertiesofM} we have $\M(\alg{2}^n) \cong \alg{L}^n$ and the natural inclusion $\prod_{i\leq n} \alg{S_i} \hookrightarrow \alg{L}^n$ yields our factorization 
$$
\begin{tikzcd}
\alg{A} \arrow[dr, hook] \arrow[r, hook]
& \M(\alg{B}) \\
& \M(\alg{2}^n) \arrow[u, hook]
\end{tikzcd}
$$
This concludes the proof.            
\end{proof} 

In particular, if $\alg{L}$ is primal, we get an explicit categorical equivalence witnessing Hu's theorem.  

\begin{corollary}{\cite{Hu1969}}\label{cor:hu}
If $\alg{L}$ is primal, then $\B \dashv \M$ yields a categorical equivalence between $\var{A}$ and $\BA$.
\end{corollary} 

We also get an algebraic analogue of Proposition \ref{prop:Stonereflectivesubcategory}(i). 

\begin{corollary}\label{cor:BAreflsubcategory}
The functor $\M$ is fully faithful and identifies $\BA$ with a reflective subcategory of $\var{A}$.
\end{corollary} 
   
By now we found detailed descriptions of most of the functors appearing in Figure \ref{fig:adjunctions}. We are still missing an algebraic understanding of the adjunction $\Q \dashv \I$. This gap is filled in the next subsection. As we will see, it is closely connected to the adjunction $\B \dashv \M$.  
\subsection{The subalgebra adjunctions}\label{subs:subalgebraadjunctions}
For every subalgebra $\alg{S}\leq \alg{L}$, there is an adjunction 
\begin{equation}\label{eq:subalgebraadjunction} 
  \xymatrix@C=35pt{ 
    \Stone \ar@/^/[rr]^{\V^{\alg{S}}}_{\bot} & & {\ \StoneL }  \ar@/^/[ll]^{\C^\alg{S}}}  
\end{equation}
which we explore in this subsection.  

The functor $\V^{\alg{S}}\colon \Stone\to\StoneL$ is given on objects by 
$$ \V^{\alg{S}}(X) = (X, \val{v}^\alg{S}) \text{ where } \forall x\in X: \val{v}^{\alg{S}}(x) = \alg{S},$$
and assigns every morphism to itself. 

The functor $\alg{C}^{\alg{S}}\colon \StoneL\to \Stone$ is given on objects by 
$$ \C^\alg{S}(X,\val{v}) = \{ x\in X \mid \val{v}(x) \leq \alg{S}\}.$$
On morphisms it acts via restriction, that is, given a morphism $m\colon (X,\val{v}) \to (Y,\val{w})$, define $m\mid_{\C^\alg{S}(X)}\colon \C^\alg{S}(X) \to \C^\alg{S}(Y)$. This is well-defined since 
$$x\in \C^\alg{S}(X,\val{v}) \Leftrightarrow \val{v}(x)\leq \alg{S} \Leftrightarrow \val{w}(m(x)) \leq \val{v}(x) \leq \alg{S} \Leftrightarrow m(x) \in \C^\alg{S}(Y,\val{w}).$$ 

Comparing this with Subsection \ref{subs:fourfunctors}, the reader may easily verify $ V^\alg{S} \dashv \C^\alg{S}$. Indeed, the adjunction $V^\alg{S} \dashv \C^\alg{S}$ generalizes the following adjunctions in Figure \ref{fig:adjunctions}:
\begin{itemize}
\item $\Vtop \dashv \U$ in the case where $\alg{S} = \alg{L}$ is the largest subalgebra of $\alg{L}$,
\item $\Vbot \dashv \C$ in the case where $\alg{S} = \langle 0,1\rangle$ is the smallest subalgebra of $\alg{L}$.
\end{itemize} 
What is special about these two extreme cases is the additional adjunction $\U \dashv \Vtop$, which `glues' the two adjunctions into the chain described in Subsection \ref{subs:fourfunctors}.

To better understand the adjunction corresponding to a subalgebra $\alg{S} \leq \alg{L}$, we dissect it into two parts as follows. 
$$
\xymatrix@C=70pt{ 
    \Stone \ar@/^/[r]^{\Vtop}_{\bot} & \Stone_\alg{S}  \ar@/^/[l]^{\U} \ar@/^/[r]^{\iota^\alg{S}}_\bot & {\ \StoneL }  \ar@/^/[l]^{(\C^\alg{S},-)}} 
$$
Here, $\iota^\alg{S}$ is the natural inclusion and the functor $(\C^{\alg{S}},-)$ is defined by 
$$ (X,\val{v}) \mapsto (\C^{\alg{S}}(X), \val{v}{\mid}_{\C^{\alg{S}}(X)})$$
on objects and, exactly like $\C^\alg{S}$, acts via restriction on morphisms. It is easy to see that this really is a decomposition of the adjunction (\ref{eq:subalgebraadjunction}), that is, 
$$ \V^{\alg{S}} = \iota^\alg{S} \circ \Vtop \text{ and } \C^{\alg{S}} = \U \circ (\C^{\alg{S}},-).$$
As before, we want to carry everything over to the algebraic side, where the dissection takes place through the subvariety 
$$ \var{A}_\alg{S} := \variety{\alg{S}}.$$  
We illustrate the entire situation in Figure \ref{fig:subalgadjunction}. Note that $\alg{S} \leq \alg{L}$ is itself semi-primal, so the semi-primal duality given by $\Sigma_\alg{S}$ and  $\Pi_{\alg{S}}$ as well as the adjunction $\B\dashv\M_\alg{S}$ make sense in this context.  
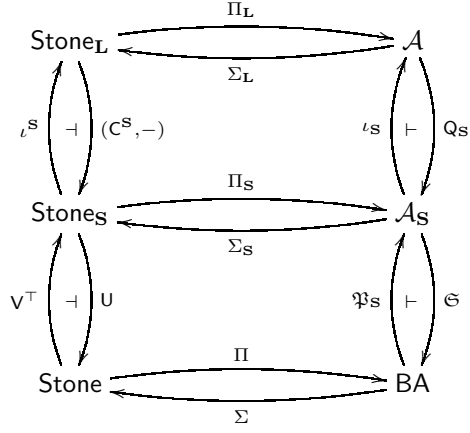
\begin{figure}[ht]
$$
\xymatrix@C=100pt@R=50pt{ 
\StoneL \ar@/^/[r]^{\Pi_\alg{L}}  
\ar@{<-}@/_8pt/[d]_{\iota^\alg{S}}^{\phantom{'}\dashv}
\ar@/^8pt/[d]^{(\C^{\alg{S}}, -)}
& 
\var{A} \ar@/^/[l]^{\Sigma_\alg{L}} 
\ar@{<-}@/_8pt/[d]_{\iota_\alg{S}}
\ar@/^8pt/[d]^{\Q_{\alg{S}}}_{\vdash\phantom{'}}
\\
\Stone_{\alg{S}} \ar@/^/[r]^{\Pi_\alg{S}}
\ar@{<-}@/_8pt/[d]_{\Vtop}^{\phantom{'}\dashv}
\ar@/^8pt/[d]^{\U}
&
\var{A}_{\alg{S}} \ar@/^/[l]^{\Sigma_\alg{S}}
\ar@{<-}@/_8pt/[d]_{\M_\alg{S}}
\ar@/^8pt/[d]^{\B}_{\vdash\phantom{'}}  
\\
\Stone \ar@/^/[r]^{\Pi}
& 
\BA \ar@/^/[l]^{\Sigma}  
}
$$ 
\caption{Dissecting the subalgebra adjunction of $\alg{S}\leq \alg{L}$.}
\label{fig:subalgadjunction}
\end{figure}      

Again, $\iota_\alg{S}$ denotes the natural inclusion. Although it may seem obvious, it is not immediate that $\iota_\alg{S}$ really is the dual of $\iota^\alg{S}$. To prove it, we make use of the following unary term, which will play an important role for the remainder of the subsection: 
$$\chi_S (x) = \bigvee_{s\in S} T_s(x).$$
On $\alg{L}$, this simply corresponds to the characteristic function of $S \subseteq L$. It is, furthermore, characteristic for the subvariety $\var{A}_\alg{S}$ in the following sense.  
\begin{lemma}\label{lem:chiS}
An algebra in $\var{A}$ is a member of $\var{A}_{\alg{S}}$ if and only if it satisfies the equation $\chi_S(x) \approx 1$.
\end{lemma}
\begin{proof}
Clearly every member of $\var{A}_\alg{S}$ satisfies the equation since $\alg{S}$ satisfies it. For the other direction, let $\alg{A} \in \var{A}$ satisfy $\chi_S(a) = 1$ for all $a\in A$. We know that $\alg{A}$ can be embedded into some $\alg{L}^I$ and for each $a\in \alg{A}$ and $i \in I$, we have $\chi_\alg{S}(\pi_i(a)) = 1$ which implies that $\pi_i(a)\in\alg{S}$. Therefore, $\alg{A}$ can be embedded into $\alg{S}^I$.    
\end{proof}
Now, let $\alg{A} \in \var{A}_\alg{S}$ and let $h\in \var{A}(\iota_\alg{S}(\alg{A}), \alg{L})$ be a homomorphism. Since $h$ preserves equations, for every $a\in A$ we get 
$$ \chi_S(a) = 1 \Rightarrow \chi_S(h(a)) = 1$$
and, therefore, $h\in \var{A}(\alg{A}, \alg{S})$. So we showed $\var{A}(\alg{A},\alg{L}) = \var{A}_\alg{S}(\alg{A},\alg{S})$ for $\alg{A}\in\var{A}_\alg{S}$, which immediately implies the following.
\begin{corollary}
The inclusion functor $\iota_\alg{S}$ is the dual of the inclusion functor $\iota^\alg{S}$.
\end{corollary} 

To complete the picture, we only need to describe the functor $\Q_\alg{S}$ from Figure \ref{fig:subalgadjunction}. Let $\alpha\colon \id_\var{A} \Rightarrow \iota_\alg{S}\circ \Q_\alg{S}$ be the unit of the adjunction $\Q_\alg{S} \dashv \iota_\alg{S}$. For any $\alg{A} \in \var{A}$, the algebra $\Q_\alg{S}(\alg{A})$ is universal for $\var{A}_\alg{S}$ in the following sense:

\emph{For every $\alg{B}\in\var{A}_\alg{S}$ and every homomorphism $f\colon \alg{A} \to \alg{B}$, there is a unique $\hat{f}\colon \Q_\alg{S}(\alg{A}) \to \alg{B}$ such that $\hat{f}\circ \alpha_\alg{A} = f$.} 
$$
\begin{tikzcd}
\alg{A} \arrow[rd, swap, "f"] \arrow[r, "\alpha_\alg{A}"] & \Q_\alg{S}(\alg{A}) \arrow[d, dashrightarrow, "\exists \hat{f}"] \\
& \alg{B}
\end{tikzcd}
$$
Therefore, the functor $\Q_\alg{S}$ may be understood as a quotient (in fact, as the largest quotient contained in $\var{A}_\alg{S}$). There is a well-known connection between quotients and equations introduced by Banaschewski and Herrlich in \cite{BanaschewskiHerrlich1976}. Not surprisingly, the equation corresponding to $\Q_\alg{S}$ is given by $\chi_S(x) \approx 1$, which is an easy consequence of the above discussion together with Lemma \ref{lem:chiS}. We summarize the results of this subsection as follows.

\begin{theorem}\label{thm:subadjunction}
For every subalgebra $\alg{S} \leq \alg{L}$, there is an adjunction 
\begin{equation*} 
  \xymatrix@C=35pt{ 
    \BA \ar@/^/[rr]^{\I_{\alg{S}}}_{\top} & & {\ \var{A} }  \ar@/^/[ll]^{\K_\alg{S}}}  
\end{equation*}
which can be dissected as
$$
\xymatrix@C=70pt{ 
    \BA \ar@/^/[r]^{\M_\alg{S}}_{\top} & \var{A}_\alg{S}  \ar@/^/[l]^{\B} \ar@/^/[r]^{\iota_\alg{S}}_\top & {\ \var{A} }  \ar@/^/[l]^{\Q_\alg{S}}} 
$$
where $\iota_\alg{S}$ is the natural inclusion functor of the subvariety $\variety{\alg{S}}\hookrightarrow \variety{\alg{L}}$ and $\Q_\alg{S}$ is the quotient functor corresponding to the equation $ \chi_S(x) \approx 1.$ 
\end{theorem}   

In particular, in the case where $\alg{S}$ is the smallest subalgebra of $\alg{L}$, we can recover the functors $\I = \iota_\alg{S} \circ \M_\alg{S}$ and $\Q$ from Figure \ref{fig:adjunctions} . 

\begin{corollary}
The functor $\I\colon \BA \rightarrow \var{A}$ is, up to categorical equivalence, an inclusion. The functor $\Q\colon \var{A} \to \BA$ is, up to categorical equivalence, the quotient by the equation 
$$\chi_{E}(x) \approx 1,$$
where $\alg{E} = \langle 0,1 \rangle$ is the smallest subalgebra of $\alg{L}$.
\end{corollary} 
\begin{proof}
Being the smallest subalgebra of a semi-primal algebra, $\alg{E}$ is primal. Therefore, by Corollary \ref{cor:hu}, the adjunction $\B \dashv \M_\alg{E}$ is an equivalence of categories. The statement follows from Theorem \ref{thm:subadjunction}.
\end{proof}

Clearly Corollary \ref{cor:BAreflsubcategory} holds not only for $\M$, but for all the functors $\I_\alg{S}$. Among them, $\I$ is special since it also has a right-adjoint. This yields the following algebraic version of Proposition \ref{prop:Stonereflectivesubcategory}(ii). 

\begin{corollary}
The functor $\I$ is fully faithful and identifies $\BA$ with a reflective and coreflective subcategory of $\var{A}$. 
\end{corollary}

We showed that, if a finite lattice-based algebra $\alg{M}$ is semi-primal, then there is an adjunction $\M_\alg{E} \dashv \B \dashv \M_\alg{M}$, where $\alg{E}$ is the smallest subalgebra of $\alg{M}$. In the next subsection we show that, conversely, the existence of an adjunction resembling this one fully characterizes semi-primality of a finite lattice-based algebra $\alg{M}$. 

\subsection{Characterizing semi-primality via adjunctions.}\label{subs: topologicalfunctors} 
The aim of this subsection is to find sufficient conditions for an adjoint of $\M_\alg{M}$ to imply semi-primality of the algebra $\alg{M}$. We will then show that, in particular, these conditions are consequences of $\U$ and $\B$ from Figure \ref{fig:adjunctions} being (essentially) \emph{topological functors}.  

Recall that, in Definition \ref{def:Boolpowerfunctor}, the Boolean power functor $\M_\alg{M}\colon \BA \to \variety{\alg{M}}$ was defined for arbitrary finite algebras $\alg{M}$. Of course, if $\alg{S}$ is a subalgebra of $\alg{M}$, then $\M_\alg{S}$ can also be seen as a functor into $\variety{\alg{M}}$, and in the following there is no need to distinguish between these two functors in our notation. The functor $\M_\alg{M}$ is faithful (unless $\alg{M}$ is trivial), but it is usually not full. In fact, it is easy to see that $\M_\alg{M}$ can only be full if $\alg{M}$ does not have any non-trivial automorphisms.      

In the main theorem of this subsection we show that, if $\M_\alg{M}$ is full and has a left-adjoint resembling $\B$, then a lattice-based algebra $\alg{M}$ is semi-primal.   
 
\begin{theorem}\label{thm:adjunctionSemiCharact}
Let $\alg{M}$ be a finite lattice-based algebra. Then $\alg{M}$ is semi-primal if and only if $\M_\alg{M}$ is full and there is a faithful functor $\mathfrak{s} \colon \variety{\alg{M}} \to \BA$ which satisfies
$$\M_{\alg{E}} \dashv \mathfrak{s} \dashv \M_{\alg{M}},$$
where $\alg{E} = \langle 0,1 \rangle$ is the smallest subalgebra of $\alg{M}$. 
\end{theorem}  

\begin{proof}
If $\alg{M}$ is semi-primal, then $\M_\alg{M}$ is full since it is dual to the full functor $\Vtop$, the functor $\mathfrak{s} = \B$ is faithful since it is dual to the faithful functor $\U$ and $\M_{\alg{E}} \dashv \B \dashv \M_{\alg{M}}$ was shown in the last two subsections. 

Now for the converse, assume that $\M_\alg{M}$ is full and there is a faithful functor $\mathfrak{s}\colon \variety{\alg{M}} \to \BA$ with $\M_{\alg{E}} \dashv \mathfrak{s} \dashv \M_{\alg{M}}$. For abbreviation we write $\var{V}$ for $\variety{\alg{M}}$.  We will make use of the following properties of $\mathfrak{s}$:
\begin{enumerate}[(i)]
\item The unit $\eta\colon \id_{\var{V}} \Rightarrow \M_\alg{M} \circ \mathfrak{s}$ is a monomorphism in each component, 
\item $\mathfrak{s}$ preserves monomorphisms and finite products. 
\end{enumerate}
Condition (i) follows from $\mathfrak{s}$ being faithful and (ii) follows from $\mathfrak{s}$ being a right-adjoint. 

Our first goal is to prove the equivalence
\begin{equation}
\mathfrak{s}(\alg{A}) \cong \alg{2} \Leftrightarrow \exists \alg{S} \in \mathbb{S}(\alg{M}): \alg{A} \cong \alg{S}. 
\end{equation}       
If $\mathfrak{s}(\alg{A}) \cong \alg{2}$, use that by (i) there is an embedding $\alg{A} \hookrightarrow \M_{\alg{M}}(\mathfrak{s}(\alg{A}))$. Since $\M_{\alg{M}}(\mathfrak{s}(\alg{A})) \cong\alg{M}$, it follows that $\alg{A}$ is isomorphic to a subalgebra of $\alg{M}$. Conversely, first note that $\mathfrak{s}(\alg{M}) \cong \alg{2}$ since, using that $\M_{\alg{M}}$ is full and $\mathfrak{s} \dashv \M_{\alg{M}}$, we have
$$1 = |\BA(\alg{2}, \alg{2})| = |\var{V}(\alg{M}, \alg{M})| = |\var{V}\big(\alg{M}, \M_{\alg{M}}(\alg{2})\big)| = |\BA\big(\mathfrak{s}(\alg{M}), \alg{2})\big)|,$$
which is only possible for $\mathfrak{s}(\alg{M}) \cong \alg{2}$. 
Now if $\alg{A} \cong \alg{S} \in \mathbb{S}(\alg{M})$ then, due to (ii), the natural embedding $\alg{S} \hookrightarrow {\alg{M}}$ induces an embedding $\mathfrak{s}(\alg{S}) \hookrightarrow \mathfrak{s}(\alg{M})$. Therefore $\mathfrak{s}(\alg{S}) \cong \alg{2}$ since $\mathfrak{s}(\alg{M}) \cong \alg{2}$ does not have any proper subalgebras.

Next we show that $\alg{M}$ does not have any non-trivial internal isomorphisms. For every subalgebra $\alg{S} \in \mathbb{S}(\alg{M})$, there is a bijection between the set of Boolean homomorphisms $\mathfrak{s}(\alg{S}) \to \alg{2}$ and the set of homomorphisms $\alg{S} \to \M_{\alg{M}}(\alg{2})$. Due to (4) we have $\mathfrak{s}(\alg{S}) \cong \alg{2}$, so the former only has one element. Since $\M_{\alg{M}}(\alg{2}) \cong \alg{M}$ this means that there is only one homomorphism $\alg{S} \to \alg{M}$, namely the identity on $\alg{S}$. Every non-trivial internal isomorphism with domain $\alg{S}$ would define another such homomorphism, resulting in a contradiction. 

We now show that $\alg{M}$ is semi-primal, using the characterization of semi-primality in Proposition \ref{SP-char-squaresubalg}. That is, we want to show that $\alg{M}$ has a majority term and every subalgebra of $\alg{M}^2$ is either a product of subalgebras or the diagonal of a subalgebra of $\alg{M}$. Since $\alg{M}$ is based on a lattice, a majority term is given by the median (see the paragraph before Proposition \ref{SP-char-squaresubalg}). Let $\alg{A} \leq \alg{M}^2$ be a subalgebra and let $\iota\colon \alg{A} \hookrightarrow \alg{M}^2$ be its natural embedding. Due to (ii), this embedding induces an embedding $\mathfrak{s}(\alg{A}) \hookrightarrow \mathfrak{s}(\alg{M}^2)$ into $\mathfrak{s}(\alg{M}^2) \cong \alg{2}^2$. Therefore, either $\mathfrak{s}(\alg{A}) \cong \alg{2}^2$ or $\mathfrak{s}(\alg{A}) \cong \alg{2}$. Let $p_1\colon \alg{A} \to \alg{M}$ and $p_2\colon \alg{A} \to \alg{M}$ be $\iota$ followed by the respective projections $\alg{M}^2 \to \alg{M}$.

First assume that $p_1$ and $p_2$ coincide. Then clearly $\alg{A}$ embeds into $\alg{M}$, and therefore it is isomorphic to some subalgebra $\alg{S}$ of $\alg{M}$. Since $\alg{M}$ has no non-trivial internal isomorphisms, $\alg{A}$ needs to coincide with the diagonal of $\alg{S}$.

If $p_1$ and $p_2$ are distinct then, using that $\mathfrak{s}$ is faithful, the morphisms $\mathfrak{s}p_1\colon \mathfrak{s}(\alg{A}) \to \alg{2}$ and $\mathfrak{s}p_2 \colon \mathfrak{s}(\alg{A}) \to \alg{2}$ are distinct as well. This implies that $\mathfrak{s}(\alg{A}) \cong \alg{2}^2$. Using the adjunction $\M_{E} \dashv \mathfrak{s}$ we get 
$$ 4 = |\BA(\alg{2}^2, \mathfrak{s}(\alg{A}))| = |\var{V}(\alg{E}^2, \alg{A})| \text{ and } 4 = |\BA(\alg{2}^2, \mathfrak{s}(\alg{M}^2))| = |\var{V}(\alg{E}^2, \alg{M}^2)|.$$
So there are exactly four distinct homomorphisms $\alg{E}^2 \to \alg{A}$ and, since $\iota$ is a monomorphism, their compositions with $\iota$ are also four distinct homomorphisms $\alg{E}^2 \to \alg{M}^2$. Therefore every of the former homomorphisms arises in such a way. In particular, the natural embedding $\alg{E}^2 \hookrightarrow \alg{M}^2$ arises in this way, which implies $(0,1) \in \alg{A}$ and $(1,0) \in \alg{A}$. As noted in \cite{DaveySchumann1991}, this leads to $\alg{A} = p_1(\alg{A}) \times p_2(\alg{A})$, since whenever $(a,b), (c,d) \in \alg{A}$ we also have 
$$(a,d) = \big((a,b) \wedge (1, 0)\big) \vee \big( (c,d) \wedge (0,1) \big) \in \alg{A}.$$
This concludes the proof.            
\end{proof}

In the remainder of this subsection we show how this theorem relates to the theory of \emph{topological functors} (see, \emph{e.g.}, \cite[Chapter VI.21]{AdamekHerrlichStrecker1990} or \cite[Chapter 7]{Borceux1994}).      
Intuitively speaking, topological functors behave similarly to the forgetful functor $\mathsf{Top} \to \Set$ out of the category of all topological spaces. Still, the definitions involved are rather technical and the reader not familiar with this topic may skip this part.  

\begin{definition}\label{def:topolFunctor}
We call a functor $\func{F} \colon \var{C} \to \var{D}$ 
\begin{enumerate}
\item \emph{topological} if it is faithful and every $\func{F}$-structured source has an initial lift,
\item \emph{essentially topological} if it is topological up to categorical equivalence of $\var{C}$ and $\var{D}$.  
\end{enumerate}           
\end{definition}  
The need for this distinction arises because certain properties of topological functors, \emph{e.g.}, \emph{amnesticity} \cite[Definition 3.27]{AdamekHerrlichStrecker1990}, are not preserved under categorical equivalence (this issue is addressed in \cite{nlab:topological_concrete_category}).  

The following is our key observation for the last part of this subsection.

\begin{proposition}\label{prop:Bistopological}
The forgetful functor $\U \colon \StoneL \to \Stone$ is topological and the Boolean skeleton functor $\B \colon \var{A} \to \BA$ is essentially topological.   
\end{proposition} 

\begin{proof}
We only need to show that $\U$ is topological, which immediately implies that $\B$ is essentially topological due to \cite[Theorem 21.9]{AdamekHerrlichStrecker1990} together with the fact that $\B$ is naturally isomorphic to the dual of $\U$. 

Of course $\U$ is faithful since it is the identity on morphisms. Now let $X\in \Stone$ and let $(f_i\colon X \to \U(X_i, \val{v}_i))_{i\in I}$ be a $\U$-structured source (i.e., a collection of continuous maps) indexed by a class $I$. We define $\val{v}\colon X \to \mathbb{S}(\alg{L})$ by
$$ \val{v}(x) = \bigvee_{i\in I} \val{v}_i(f_i(x)).$$
Note that this is well-defined, since $\mathbb{S}(\alg{L})$ is finite and that $(X,\val{v})$ is a member of $\StoneL$, since $\val{v}^{-1}(\alg{S}{\downarrow}) = \bigcap_{i\in I} f_i^{-1} (\val{v}_i^{-1}(\alg{S}{\downarrow}))$ is closed. Every $f_i$ is now also a morphism in $\StoneL$, which defines a lift of the source. To show that it is initial, assume there are $\StoneL$-morphisms $(g_i \colon (Y,\val{w}) \to (X_i, \val{v}_i))_{i\in I}$ and a continuous map $g\colon Y \to X$ with $f_i \circ g = g_i$. All we need to show is that $g$ defines a $\StoneL$-morphism $(Y,\val{w}) \to (X,\val{v})$. To see this simply note that 
$$ \val{v}(g(y)) = \bigvee_{i\in I} \val{v}_i\big(f_i(g(y))\big) = \bigvee_{i\in I}\val{v}_i(g_i(y)) \leq \val{w}(y),$$
which concludes the proof.      
\end{proof} 

We can now easily show the following. 

\begin{corollary}\label{cor:semiCharTopFunctor}
Let $\alg{M}$ be a finite lattice-based algebra. Then $\alg{M}$ is semi-primal if and only if there is an essentially topological functor $\mathfrak{s} \colon \variety{\alg{M}} \to \BA$ which satisfies
$$\M_{\alg{E}} \dashv \mathfrak{s} \dashv \M_{\alg{M}},$$
where $\alg{E} = \langle 0,1 \rangle$ is the smallest subalgebra of $\alg{M}$. 
\end{corollary}

\begin{proof}
In the previous proposition we showed that if $\alg{M}$ is semi-primal, then $\B$ is essentially topological. 

Conversely, if such an essentially topological $\mathfrak{s}$ exists, it is faithful by definition and both its adjoints $\M_{\alg{M}}$ and $\M_{\alg{E}}$ are full by \cite[Proposition 21.12]{AdamekHerrlichStrecker1990}. Therefore, due to Theorem \ref{thm:adjunctionSemiCharact}, $\alg{M}$ is semi-primal. 
\end{proof}
    
In this section we gained an algebraic understanding of all the functors between $\var{A}$ and $\BA$ appearing on the right-hand side of Figure \ref{fig:adjunctions}. Furthermore, we now showed how properties of the Boolean skeleton functor $\B$ characterize semi-primality. In the next section we investigate how canonical extensions of algebras in $\var{A}$ behave under these functors. One of the main results is that the Boolean skeleton functor $\B$ may be used to identify canonical extensions of algebras in $\var{A}$.  

\section{Discrete duality and canonical extensions}\label{sec:canext}
In this section we describe a semi-primal discrete duality similar to the well-known discrete duality between $\Set$ and $\CABA$, the category of complete atomic Boolean algebras with complete homomorphisms. It can be obtained from the finite duality in a similar way to the one of Section \ref{sec:semi-primal duality}, except that now we lift it to the level of $\Ind(\SetL^\omega)$ and $\Pro(\var{A}^\omega)$. The members of the latter category are known to be precisely the \emph{canonical extensions} \cite{GehrkeJonsson2004} of members of $\var{A}$ (see \cite{DaveyPriestley2012}), and we will provide two new characterizations of this category (Corollary \ref{cor:ProAIPSL} and Theorem \ref{thm:ProACAA}). Lastly we show that, as in the primal case $\alg{L} = \alg{2}$, the topological duality from Section \ref{sec:semi-primal duality} can be connected to its discrete version via an analogue of the \emph{Stone-Čech compactification}.    

Our first goal is to identify $\Ind(\SetL^\omega)$. Although it may not be surprising, it will still take some work to prove that it can be identified with the following category.

\begin{definition}
The category $\SetL$ has objects of the form $(X,v)$ where $X\in \Set$ and $v\colon X\to \mathbb{S}(\alg{L})$ is an arbitrary map. A morphism $m\colon (X,v) \rightarrow (Y,w)$ is a map $X\to Y$ which always satisfies
$$ w(m(x)) \leq  v(x).$$   
\end{definition}     

\begin{remark}
In the context of fuzzy sets, Goguen \cite{Goguen1967, Goguen1974} initiated the study of such categories. This research was continued, \emph{e.g.}, in \cite{Barr1986, Walker2004}. In this remark we stick to the notation of \cite{Goguen1974}. Given a complete lattice $\mathcal{V}$, the category $\Set(\mathcal{V})$ of $\mathcal{V}$-fuzzy sets has objects $(X,A)$ where $A\colon X\to \mathcal{V}$. Morphisms $(X,A) \to (X',A')$ are maps $m\colon X\to Y$ which satisfy $A'(m(x)) \geq A(x)$ for all $x\in X$. In the context of fuzzy set theory, people were mainly interested in the case where $\mathcal{V} = [0,1]$. However, we retrieve $\SetL$ in the case where $\mathcal{V}$ is the order-dual of $\mathbb{S}(\alg{L})$. \hfill $\blacksquare$         
\end{remark}

Since we are interested in the $\Ind$-completion of $\SetL^\omega$, we will first discuss (filtered) colimits in this category.

\begin{lemma}\label{lem:coprodSetL}
The category $\SetL$ is cocomplete. The colimit $\mathrm{colim}_{i\in I}(X_i,v_i)$ of a filtered diagram $\big(f_{ij}\colon (X_i, v_i) \to (X_j, v_j)\mid i \leq j\big)$ is realized by $\big((\coprod_{i\in I}X_i){/}{\sim} , \bar{v}\big)$. Here, for $x_i\in X_i$ and $x_j\in X_j$,
$$ x_i \sim x_j \iff \exists k \geq i, j: f_{ik}(x_i) = f_{jk}(x_j)$$
and 
$$ \bar{v}([x_i]) = \bigwedge_{x_i\sim x_j \in X_j} v_j(x_j),$$
where $[x_i]$ is the equivalence class of $x_i$ with respect to $\sim$.   
\end{lemma}

\begin{proof}
The proof that $\SetL$ is cocomplete is completely analogous to the one in \cite{Walker2004}. 
For filtered colimits, on the underlying level of $\Set$ we know that $X :=\coprod_{i\in I}(X_i)/{\sim}$ with the canonical inclusions $\rho_i\colon X_i \to X$ is the colimit of the diagram. To see that all the $\rho_i$ are morphisms in $\SetL$ note 
$$ \bar{v}(\rho_i(x_i)) = \bigwedge_{x_i \sim x_j\in X_j} v_j(x_j) \leq v_i(x_i).$$
Given another cocone $\gamma_i \colon (X_i, v_i) \to (Z,u)$, the unique map $g\colon X\to Z$ is a morphism in $\SetL$ since, for $x_i \in X_i$ and $x_i \sim x_j \in X_j$ we have $u\big(g(\rho_j(x_j))\big) = u(\gamma_j(x_j)) \leq v_j(x_j)$ and thus 
$$ u\big(g([x_i])\big) \leq \bigwedge_{x_i \sim x_j\in X_j} v_j(x_j) = \bar{v}([x_i]),$$
which concludes the proof.     
\end{proof}

We will also make use of the following general result. 

\begin{lemma}\label{lem:adjointFinPresentable}
Let $\func{F} \colon \var{C} \to \var{D}$ be a functor between categories $\var{C}$ and $\var{D}$ which both admit filtered colimits. If $\func{F}$ has a right-adjoint $\func{G}$ which preserves filtered colimits, then $\func{F}$ preserves finitely presentable objects.  
\end{lemma}

\begin{proof}
Let $C\in \var{C}$ be finitely presentable. We want to show that $F(C)$ is finitely presentable in $\var{D}$. Let $\mathrm{colim}_i D_i$ be a filtered colimit in $\var{D}$. Then 
$$ \var{D}\big(\func{F}(C), \mathrm{colim}_i D_i\big) \cong \mathrm{colim}_i\var{C}\big(C, \func{G}(D_i)\big) \cong \mathrm{colim}_i \var{D}\big(\func{F}(C), D_i\big),$$
where the first isomorphism comes from the fact that $\func{G}$ preserves filtered colimits and $C$ is finitely presentable.  
\end{proof}

\begin{corollary}
If $X$ is a finite set, then $(X,v)$ is finitely presentable in $\SetL$ for every $v\colon X\to \mathbb{S}(\alg{L})$.  
\end{corollary}

\begin{proof}
Let $X = \{ x_1, \dots ,x_n \}$ and let $v(x_i) = \alg{S_i}$. Then we can clearly identify 
$$ (X,v) \cong \coprod_{1\leq i \leq n} (\{ x_i \}, v^{\alg{S_i}}),$$
where $v^\alg{S_i}(x_i) = \alg{S_i}$.
Since filtered colimits commute with finite limits in $\Set$, it now suffices to show that all $(\{ x_i \}, v^{\alg{S_i}})$ are finitely presentable. Just like in Subsection~\ref{subs:subalgebraadjunctions} we can define the adjunction $\V^{\alg{S}} \dashv \C^{\alg{S}}$ between $\SetL$ and $\Set$ for every subalgebra $\alg{S} \leq \alg{L}$. By Lemma \ref{lem:adjointFinPresentable} it now suffices to show that $\C^{\alg{S}}$ preserves filtered colimits. So let $(X,\bar{v})$ be a filtered colimit as in Lemma \ref{lem:coprodSetL}. We know that $\C^{\alg{S}}(X) = \{ [x_i] \mid \exists x_i \sim x_j \in X_j, v_j(x_j) \leq \alg{S} \}$. Therefore, for all $[x_i] \in \C^{\alg{S}}$ we can choose representatives with $x_i \in \C^{\alg{S}}(X_i, v_i)$. This yields a bijection between $\C^{\alg{S}}(X)$ and $\mathrm{colim} \C^{\alg{S}}(X_i,v_i)$.       
\end{proof}

We now have everything at hand to easily prove the following. 

\begin{theorem}\label{thm:IndSetL}
$\Ind(\SetL^\omega)$ is categorically equivalent to $\SetL$. 
\end{theorem}

\begin{proof}
Since $\SetL$ is cocomplete, the inclusion $\iota\colon \SetL^\omega \to \SetL$ has a unique finitary extension $\hat{\iota}\colon \Ind(\SetL^\omega) \to \SetL$. Since $\iota$ is fully faithful and, by the above corollary, maps all objects to finitely presentable objects in $\SetL$, this extension is also fully faithful. To see that it is essentially surjective note that, just like in $\Set$, every member of $\SetL$ is the filtered colimit of its finite subsets. 
\end{proof}

We now take a closer look at the category $\Pro(\var{A}^\omega)$. It is well-known that it consists of the \emph{canonical extensions} \cite{GehrkeJonsson2004} of algebras in $\var{A}$. In \cite{DaveyPriestley2012} a description of these canonical extensions as topological algebras can be found. But, as in the case of complete atomic Boolean algebras $\CABA \simeq \mathbb{I}\mathbb{P}(\alg{2})$, this need not be the only description. In the following we apply results of Section \ref{sec:adcuntions} to find two easy alternatives. The first one is in terms of (arbitrary) products of subalgebras of $\alg{L}$ with complete homomorphisms. 

\begin{definition}
Let $\hat{\var{A}}$ be the category with algebras from $\mathbb{I}\mathbb{P}\mathbb{S}(\alg{L})$ as objects and complete homomorphisms as morphisms. 
\end{definition}  

We can essentially repeat our proof of the finite duality from Corollary \ref{cor:finiteeq}, once we prove the following result analogous to Proposition \ref{prop:homsareprojections}.

\begin{proposition}
Let $\alg{A} = \prod_{i\in I}\alg{S_i} \in \hat{\var{A}}$. Then the complete homomorphisms $\alg{A} \to \alg{L}$ are precisely the projections (followed by inclusions) in each component.  
\end{proposition}  

\begin{proof}
By Proposition \ref{prop:HomeoBooleanSkeleton} there is a bijection between $\var{A}(\alg{A},\alg{L})$ and $\BA(\B(\alg{A}),\alg{2})$ given by $h\mapsto h{\mid}_{\B(\alg{A})}$. In particular, if $h$ is complete, then so is its restriction. Since $\B(\alg{A}) = \alg{2}^I$, the only complete homomorphisms $\B(\alg{A})\to \alg{2}$ are the projections, and they are the restrictions of the respective projections $\alg{A}\to \alg{L}$.  
\end{proof}

\begin{corollary}\label{cor:ProAIPSL}
$\Pro(\mathcal{A}^\omega)$ is categorically equivalent to $\hat{\var{A}}$
\end{corollary}

\begin{proof}
By Theorem \ref{thm:IndSetL} it suffices to show that $\SetL$ is dually equivalent to $\hat{\var{A}}$. This is done completely analogous to the proof of Corollary \ref{cor:finiteeq}.   
\end{proof}  

The second description of $\Pro(\var{A}^\omega)$ is in terms of the Boolean skeleton. 

\begin{definition}
The category $\CAA$ has as objects algebras $\alg{A}\in\var{A}$ which have a complete lattice-reduct and which satisfy $\B(\alg{A}) \in \CABA$. The morphisms in $\CAA$ are the complete homomorphisms. 
\end{definition} 

\begin{theorem}\label{thm:ProACAA}
$\Pro(\var{A}^\omega)$ is categorically equivalent to $\CAA$.  
\end{theorem}    

\begin{proof}
Using Corollary \ref{cor:ProAIPSL} we show that $\CAA$ is categorically equivalent to $\hat{\var{A}}$. Clearly there is a fully faithful inclusion functor $\hat{\var{A}} \hookrightarrow \CAA$. So it suffices to show that this functor is essentially surjective. In other words, we want to show that every object of $\CAA$ is isomorphic to a product of subalgebras of $\alg{L}$. 

So consider $\alg{A} \in \CAA$. Since the adjunction $\B \dashv \M$ restricts to $\CABA$ and $\CAA$, we can use Corollary \ref{cor:embedAinMBA} to get a \emph{complete} embedding $\eta_\alg{A}\colon\alg{A}\hookrightarrow \M(\B(\alg{A}))$. Since $\B(\alg{A})$ is in $\CABA$ it is isomorphic to $\alg{2}^I$ for some index set $I$. Thus $\M(\B(\alg{A}))\cong \M(\alg{2^I}) \cong \alg{L}^I$. We show that $\alg{A}$ is isomorphic to the direct product of subalgebras $\prod_{i\in I} \pr_i(\eta_\alg{A}(A))$. For this it suffices to show that the injective homomorphism $\eta_\alg{A}$ maps onto it. So let $\alpha$ be an element of this product. For each $i\in I$ choose $a_i \in \alg{A}$ such that $\pr_i(\eta_\alg{A}(a_i)) = \alpha(i)$. Since $\alg{2}^I \cong \B(\alg{A}) \subseteq \alg{A}$ all atoms $b_i \in \alg{2}^I$ (defined by $b_i(j) = 1$ iff $j = i$) can be considered as members of $\alg{A}$. Now define 
$$ a = \bigvee \{ a_i \wedge b_i \mid i\in I \}.$$ 
Since $\alg{A}$ is complete, we have $a\in \alg{A}$. And since $\eta_\alg{A}$ is a complete homomorphism we have $\eta_\alg{A}(a) = \alpha$ (because $\pr_i(\eta_\alg{A}(a)) = \eta_\alg{A}(a_i) = \alpha(i)$).        
\end{proof}

With the results from this section thus far, it is clear that the chains of adjunctions from Section \ref{sec:adcuntions} (summarized in Figure \ref{fig:adjunctions}) have their discrete counterparts, equally defined, between $\SetL$ and $\Set$ and $\CAA$ and $\CABA$, respectively. To make the connection between Figure \ref{fig:adjunctions} and its discrete counterpart, we finish this section by connecting the respective dualities as indicated in Figure \ref{fig:compactification}.  
\begin{figure}[ht]
$$
\xymatrix@C=100pt@R=50pt{ 
\StoneL \ar@/^/[r] 
\ar@{<-}@/_8pt/[d]_{\beta_\alg{L}}^{\phantom{'}\dashv}
\ar@/^8pt/[d]^{(-)^\flat}
& 
\var{A} \ar@/^/[l] 
\ar@{<-}@/_8pt/[d]_{\iota_{c}}
\ar@/^8pt/[d]^{(-)^\delta}_{\vdash\phantom{'}}
\\
\SetL \ar@/^/[r]
&
\CAA \ar@/^/[l] 
}
$$ 
\caption{Compactification and canonical extension.}
\label{fig:compactification}
\end{figure}
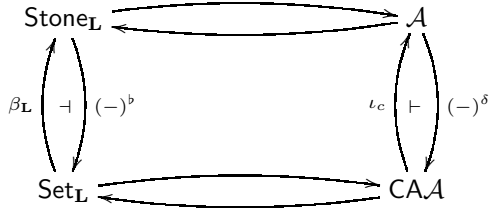  

Here $(-)^\flat\colon \StoneL\to \SetL$ is the forgetful functor with respect to topology and $\iota_c\colon \CAA\to \var{A}$ is the obvious inclusion functor (note that both these functors are not full). The functor $(-)^\delta$ takes an algebra to its canonical extension. In the primal case $\alg{L} = \alg{2}$, it is well-known that $\beta_\alg{2} =: \beta$ is the \emph{Stone-Čech compactification} (see, \emph{e.g.}, \cite[Section IV.2]{Johnstone1982}). This has been generalized to the \emph{Bohr compactification} in a (much broader) framework which includes ours in \cite{DaveyHaviar2017}. However, since things are particularly simple in our setting, we directly show how to define $\beta_\alg{L}$.

Given $(X,v)\in \SetL$, there is a natural way to extend $v$ to the Stone-Čech compactification $\beta(X)$ of $X$. Indeed, since $v\colon X\to \mathbb{S}(\alg{L})$ can be thought of as a continuous map between discrete spaces, by the universal property of $\beta$ it has a unique continuous extension $\val{\tilde{v}}\colon \beta(X)\to \mathbb{S}(\alg{L})$. Here, $\val{\tilde{v}}^{-1}( \alg{S}{\downarrow} )$ is given by the topological closure of $v^{-1}(\alg{S}{\downarrow})$ in $\beta(X)$. Thus, for every morphism $f \colon (X,v) \to (Y,w)$ in $\SetL$, the continuous map $\beta f$ defines a morphism $(\beta(X), \val{\tilde{v}}) \to (\beta(Y),\val{\tilde{w}})$ in $\StoneL$. This is due to the observation that whenever $x\in \val{\tilde{v}^{-1}( S{\downarrow})} = \overline{v^{-1}( \alg{S}{\downarrow})}$, by continuity of $\beta f$ and the morphism property of $f$, we have $\beta f(x) \in \overline{w^{-1}(\alg{S}{\downarrow})} = \val{\tilde{w}}^{-1}(\alg{S}{\downarrow})$.  

\begin{proposition}\label{prop:StoneCech}
The functor $\beta_\alg{L}\colon \SetL \to \StoneL$ defined on objects by $$\beta_\alg{L}(X,v) = (\beta(X), \val{\tilde{v}})$$ and by $f \mapsto \beta f$ on morphisms is the dual of the canonical extension functor $(-)^\delta \colon \var{A}\to \CAA$.
\end{proposition}    

\begin{proof}
It suffices to show that $\beta_{\alg{L}}$ satisfies the following universal property. Given $(Y,\val{w})\in \StoneL$, every $\SetL$-morphism $f\colon (X,v) \to (Y,\val{w})$ extends uniquely to a $\StoneL$-morphism $\tilde{f}\colon (\beta(X), \val{\tilde{v}})\to (Y,\val{w})$. On the levels of $\Set$ and $\Stone$ we get a unique continuous extension $\tilde{f}$. To show it is a $\StoneL$-morphism, similarly to before, note that if $x\in \overline{v^{-1}( S{\downarrow})}$, then by continuity  
$$\tilde{f}(x) \in \overline{f\big(v^{-1}( \alg{S}{\downarrow} )\big)} \subseteq \overline{\val{w}^{-1}(\alg{S}{\downarrow})}.$$
Since $\val{w}^{-1}(\alg{S}{\downarrow})$ is closed it equals its own closure. This concludes the proof. 
\end{proof}  

This nicely wraps up this paper by connecting all of its main sections. In the last section we give a quick summary and discuss some possible directions of future research along similar lines.     
 
\section{Concluding Remarks and Further Research}\label{sec:conclusion}
We explored semi-primality by means of category theory, showing how a variety generated by a semi-primal lattice expansion relates to the variety of Boolean algebras. Various adjunctions provide insight into the many similarities between these varieties. A schematic summary of our results can be found in Figure \ref{fig:summary}, which also emphasizes once more how close $\BA$ and $\var{A}$ really are.   

\begin{figure}[ht]
\begin{tikzcd}
& \SetL^\omega \arrow[ddl, dashed, near start, "{\Pro}" description] \arrow[dr, dashed, "{\Ind}" description] \arrow[rrr, leftrightarrow] & & & \var{A}^\omega \arrow[dr, dashed, "{\Pro}" description] \\
& & \SetL \arrow[from = dd, shift left = 3]\arrow[dd, shift right = 1]\arrow[from = dd, shift right = 1] \arrow[dd, shift left = 3] \arrow[rrr, leftrightarrow] & & & \mathsf{CA}\var{A} \arrow[from = dd, shift left = 3]\arrow[dd, shift right = 1]\arrow[from = dd, shift right = 1] \arrow[dd, shift left = 3] \\
\StoneL \arrow[from = dd, shift left = 3]\arrow[dd, shift right = 1]\arrow[from = dd, shift right = 1] \arrow[dd, shift left = 3] \arrow[rru, shift right = 1] \arrow[from = rru, shift right = 1] \arrow[rrr, leftrightarrow, crossing over] & & & \var{A} \arrow[rru, shift right = 1] \arrow[from = rru, shift right = 1] \arrow[from = uur, crossing over, dashed, near start, "{\Ind}" description] & \\
& & \Set \arrow[from = ddl, dashed, near start, "{\Ind}" description] \arrow[rrr, leftrightarrow] & & & \CABA \\
\Stone \arrow[rru, shift right = 1] \arrow[from = rru, shift right = 1] & & & \BA \arrow[from = lll, leftrightarrow, crossing over] \arrow[uu, shift left = 3, crossing over]\arrow[from = uu, shift right = 1, crossing over] \arrow[uu, shift right = 1, crossing over]\arrow[from = uu, shift left = 3, crossing over]  \arrow[rru, shift right = 1] \arrow[from = rru, shift right = 1] & \\
& \Set^\omega \arrow[lu, dashed, "{\Pro}" description] \arrow[rrr, leftrightarrow] & & & \BA^\omega \arrow[lu, dashed, "{\Ind}" description] \arrow[uur, dashed, near start, "\Pro" description]
\end{tikzcd}
 \caption{Summary of our results.}
 \label{fig:summary}
\end{figure}
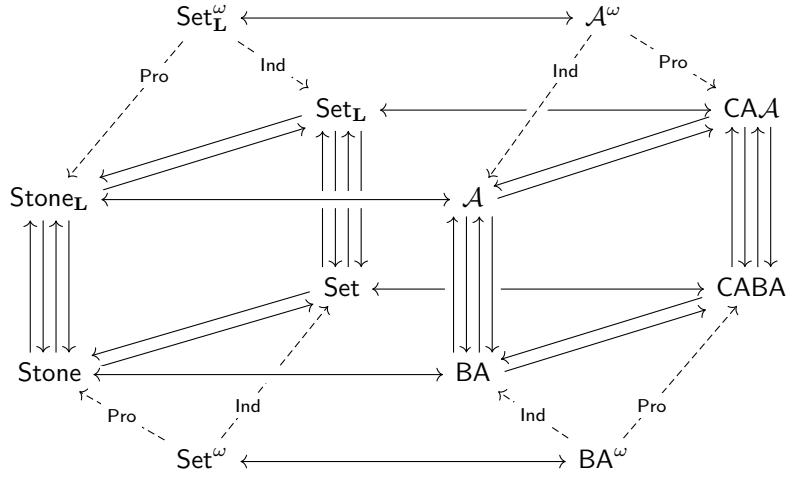

We plan to follow up this research by developing a coalgebraic framework for modal extensions of the many-valued logic corresponding to a semi-primal variety. As mentioned before, from this point of view it is reasonable to assume that $\alg{L}$ is based on a lattice. However, it seems plausible that our results generalize to the slightly more general case of semi-primal algebras which possess a \emph{coupling} in the sense of \cite{Foster1967}, essentially since Proposition \ref{SP-char-Ts} and Theorem \ref{SPDuality} still apply to this case.

We will now sketch some more potential ways to follow up this research. In general, we hope to have set an example in exploring concepts in universal algebra through the lens of (mostly elementary) category theory. 

For example, other variants of primality (see, \emph{e.g.}, \cite{Quackenbush1979, KaarliPixley2001}) could be investigated in a similar manner.
\begin{definition}\label{def:variantsofprimality}
A finite algebra $\alg{M}$ is called 
\begin{enumerate}
\item \emph{demi-semi-primal}  if it is quasi-primal and every internal isomorphism of $\alg{M}$ can be extended to an automorphism of $\alg{M}$ (see \cite{Quackenbush1971}).  
\item \emph{demi-primal} if it is quasi-primal and has no proper subalgebras (see  \cite{Quackenbush1971}).
\item \emph{infra-primal} if it is demi-semi primal and every internal isomorphism is an automorphism on its domain (see \cite{Foster1969}). 
\item \emph{hemi-primal} if every operation on $\alg{M}$ which preserves congruences is term-definable in $\alg{M}$ (see \cite{Foster1970}). 
\end{enumerate}
\end{definition} 
\begin{question}
What is the categorical relationship between $\BA$ and the variety generated by an algebra which is quasi-primal or which satisfies one of the properties of Definition \ref{def:variantsofprimality}? What about the relationship between distinct variations of primality to each other? 
\end{question}  
For quasi-primal algebras (and thus, in particular, for algebras satisfying (1), (2) or (3)), there is the duality theorem by Keimel-Werner \cite{KeimelWerner1974} (which is also a natural duality \cite{ClarkDavey1998}), possibly a good starting point to a discussion similar to the one presented here.

Hemi-primality seems to have received less attention. To the best of the authors knowledge, no duality for varieties generated by hemi-primal algebras is known thus far.   
\begin{question}
Is it possible to obtain a duality for hemi-primal varieties, for example one which stems from a finite dual equivalence using methods similar to our proof of semi-primal duality in Section \ref{sec:semi-primal duality}?   
\end{question}
The Boolean power functor $\M_\alg{M}\colon \BA \to \variety{\alg{M}}$ was defined for an arbitrary finite algebra $\alg{M}$. In the light of our results from Section \ref{sec:adcuntions}, the following question arises.  
\begin{question}
Under which circumstances does the functor $\M_\alg{M}$ have a left-adjoint? Which information about $\alg{M}$ can be retrieved from properties of the functors of the form $\M_\alg{S}$ with $\alg{S}\leq \alg{M}$?   
\end{question}     

If we consider this work as not only comparing varieties but \emph{comparing dualities}, another range of questions appears. 

\begin{question}
What is the category theoretical relationship between different dual equivalences? For example, one could consider \emph{Priestley duality} \cite{Priestley1970} or \emph{Esakia duality} \cite{Esakia1974}.    
\end{question}  

Lastly, another category theoretical approach to universal algebra, which has not been discussed in this paper, is given by Lawvere theories. For example, Hu's theorem has been analyzed from this angle in \cite{Porst2000}. Of course, one could also try to find out more about other variants of primality in this context. 
\begin{question}
How can semi-primality and other variants of primality be expressed in terms of Lawvere theories?  
\end{question}

\appendix
\section{Some semi-primal \texorpdfstring{$\FLew$}-algebras}\label{Appendix}
Here we go into more detail in some claims made in Subsection \ref{example:reslattices}. We provide examples of semi-primal $\FLew$-algebras, both chain-based and non chain-based, including the proof of semi-primality for each one of them. All of the examples and their labels are taken from the list \cite{GalatosJipsen2017} by Galatos and Jipsen. For simplicity we only discus $\FLew$-algebras without any idempotent elements other than $0$ and $1$. Due to Corollary $\ref{cor:QPFLEW}$ they are all quasi-primal. To prove semi-primality, by Proposition \ref{SP-char-discriminator}, it suffices to describe all subalgebras and argue why there can't be any non-trivial isomorphisms between then. 

We begin with the quasi-primal $\FLew$-chains of five elements $R^{5,1}_{1,17}$ to $R^{5,1}_{1,22}$ in \cite[p.2, row 2]{GalatosJipsen2017} depicted in Figure \ref{fig:QPFLewChains}.

\begin{figure}[ht]
\begin{tikzpicture}
  \node (1) at (0,3.6) {$1$};
  \node (a) at (0,2.7) {$a$};
  \node (b) at (0,1.8) {$b$};
  \node (c) at (0,0.9) {$c$};
  \node (0) at (0,0) {$a^2$};
  \node (label) at (0, -1) {$R^{5,1}_{1,17}$};
  \draw (1) -- (a) -- (b) -- (c) -- (0);
\end{tikzpicture}
\hspace{0.5cm}
\begin{tikzpicture}
  \node (1) at (0,3.6) {$1$};
  \node (a) at (0,2.7) {$a$};
  \node (b) at (0,1.8) {$b$};
  \node (c) at (0,0.9) {$c = a^2$};
  \node (0) at (0,0) {$ab$};
  \node (label) at (0, -1) {$R^{5,1}_{1,18}$};
  \draw (1) -- (a) -- (b) -- (c) -- (0);
\end{tikzpicture}
\hspace{0.5cm}
\begin{tikzpicture}
  \node (1) at (0,3.6) {$1$};
  \node (a) at (0,2.7) {$a$};
  \node (b) at (0,1.8) {$b = a^2$};
  \node (c) at (0,0.9) {$c$};
  \node (0) at (0,0) {$ab$};
  \node (label) at (0, -1) {$R^{5,1}_{1,19}$};
  \draw (1) -- (a) -- (b) -- (c) -- (0);
\end{tikzpicture}
\hspace{0.5cm}
\begin{tikzpicture}
  \node (1) at (0,3.6) {$1$};
  \node (a) at (0,2.7) {$a$};
  \node (b) at (0,1.8) {$b$};
  \node (c) at (0,0.9) {$c = a^2 = ab$};
  \node (0) at (0,0) {$b^2 = ac$};
  \node (label) at (0, -1) {$R^{5,1}_{1,20}$};
  \draw (1) -- (a) -- (b) -- (c) -- (0);
\end{tikzpicture}
\hspace{0.5cm}
\begin{tikzpicture}
  \node (1) at (0,3.6) {$1$};
  \node (a) at (0,2.7) {$a$};
  \node (b) at (0,1.8) {$b = a^2$};
  \node (c) at (0,0.9) {$c = ab$};
  \node (0) at (0,0) {$b^2 = ac$};
  \node (label) at (0, -1) {$R^{5,1}_{1,21}$};
  \draw (1) -- (a) -- (b) -- (c) -- (0);
\end{tikzpicture}
\hspace{0.5cm}
\begin{tikzpicture}
  \node (1) at (0,3.6) {$1$};
  \node (a) at (0,2.7) {$a$};
  \node (b) at (0,1.8) {$b$};
  \node (c) at (0,0.9) {$c= a^2 = b^2$};
  \node (0) at (0,0) {$ac$};
  \node (label) at (0, -1) {$R^{5,1}_{1,22}$};
  \draw (1) -- (a) -- (b) -- (c) -- (0);
\end{tikzpicture}
 \caption{The quasi-primal $\FLew$-chains of order five.}
 \label{fig:QPFLewChains}
\end{figure}
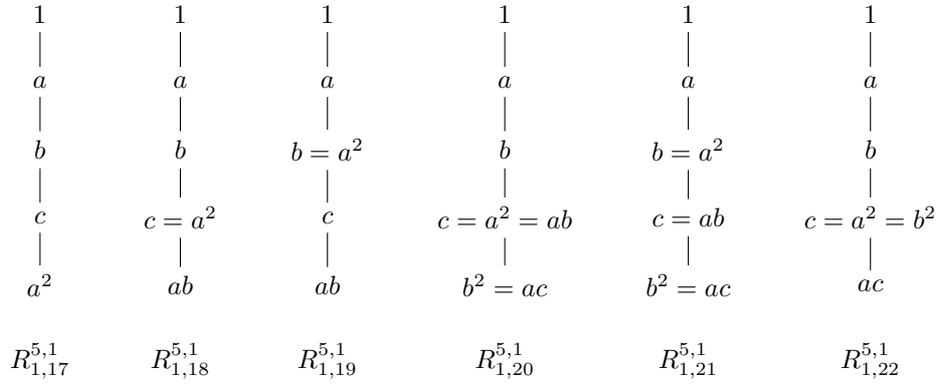 
\begin{claim}
Except for the first one, all algebras depicted in Figure \ref{fig:QPFLewChains} are semi-primal. 
\end{claim}
\begin{proof}
$R^{5,1}_{1,17}$ is not semi-primal because it has isomorphic subalgebras $\{ 0,1, a, b \}$ and $\{ 0,1, a, c \}$.  

In the following we show why the other ones are semi-primal by describing the subalgebras other than the obvious ones  $\{ 0,1 \}$ and $\{ 0,1,a,b,c\}$. Since isomorphisms need to be order-preserving, it suffices to note that there are never two subalgebras of the same size in the examples below. 

$R^{5,1}_{1,18}$: There are no other subalgebras since $\{ \neg a, a^2 \} = \{ b, c \} \subseteq \langle  a \rangle$ and $\neg b = \neg c = a$, thus $a\in \langle b \rangle$ and $a\in \langle c \rangle$. \\
$R^{1,5}_{1,19}$: There is the subalgebra $\langle a \rangle = \langle b \rangle = \{ 0,1, a, b \}$ since $a\rightarrow b = a$, $\neg a = b$ and $\neg b = a$. Since $a = \neg c$ we have $a\in \langle c \rangle$, so $c$ generates the entire algebra. 

$R^{5,1}_{1,20}$: There are two different sized subalgebras $\langle a \rangle = \langle c \rangle = \{0,1,a,c\}$ (since $\neg a = c, \neg c = a$ and $a\rightarrow c = a$)  and $\langle b \rangle = \{ 0,1, b \}$ (since $\neg b = b\rightarrow b = b$) 

$R^{5,1}_{1,21}$: Note that this algebra corresponds to the \L ukasiewicz-chain $\lucas_4$. As thus expected, there is the subalgebra $\langle b \rangle = \{ 0,1, b \}$, while $b\in \langle a \rangle \cap \langle c \rangle$ since $a = \neg c, c = \neg a$ and $b = a^2$. 

$R^{5,1}_{1,22}$: There is the subalgebra $\langle a \rangle = \langle c \rangle = \{ 0,1,a,c \}$ (since $\neg a = c$, $\neg c = a$ and $a \rightarrow c = a$). Since $\neg b = c$ and $\neg c = a$ we find that $b$ generates the entire algebra.   
\end{proof}

To also provide non chain-based examples, we examine the $\FLew$-algebras $R^{6,2}_{1,11}$ (\cite[p.19, row 5]{GalatosJipsen2017}) and $R^{6,3}_{1,9}$ (\cite[p.21, row 3]{GalatosJipsen2017}) depicted in Figure \ref{fig:QPFLewAlgs}.

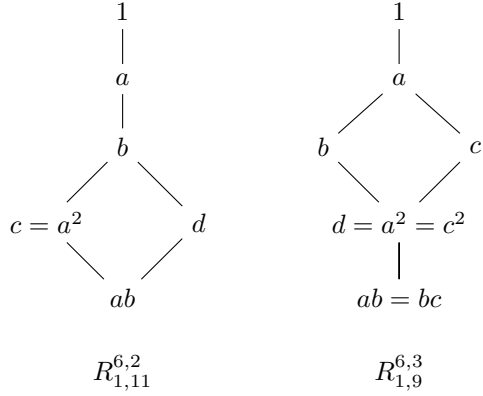
\begin{figure}[ht]
\begin{tikzpicture}
  \node (1) at (0,3.8) {$1$};
  \node (a) at (0,2.9) {$a$};
  \node (b) at (0,2) {$b$};
  \node (c) at (-1,1) {$c = a^2$};
  \node (d) at (1,1) {$d$};
  \node (0) at (0,0) {$ab$};
  \node (label) at (0, -1) {$R^{6,2}_{1,11}$};
  \draw (1) -- (a) -- (b) -- (c) -- (0) -- (d)-- (b);
\end{tikzpicture}
\hspace{1cm}
\begin{tikzpicture}
  \node (1) at (0,3.8) {$1$};
  \node (a) at (0,2.9) {$a$};
  \node (b) at (-1,2) {$b$};
  \node (c) at (1,2) {$c$};
  \node (d) at (0,1) {$d = a^2 = c^2$};
  \node (0) at (0,0) {$ab = bc$};
  \node (label) at (0, -1) {$R^{6,3}_{1,9}$};
  \draw (1) -- (a) -- (b) -- (d) -- (0) -- (d) -- (c) -- (a);
\end{tikzpicture}
 \caption{Two semi-primal $\FLew$-algebras of order six.}
 \label{fig:QPFLewAlgs}
\end{figure}

\begin{claim}
The two $\FLew$-algebras depicted in Figure \ref{fig:QPFLewAlgs} are semi-primal.
\end{claim}

\begin{proof}
$R^{6,2}_{1,11}$: The only possible candidate for an automorphism of this algebra is the bijection $f$ exchanging $c$ and $d$ (since it needs to be order-preserving). This map, however, is not a homomorphism, as witnessed by the fact that $f(a^2) = f(c) = d$ while $f(a)^2 = a^2 = c$. The only other subalgebra other than $\{ 0,1 \}$ is $\langle a \rangle = \{ 0,1, a,b, c \}$ since we have $\neg a = b$, $a^2 = c$, $\neg c = a$, $a\rightarrow b = a$, $a \rightarrow c = a$ and $b \rightarrow c = a$. Since this subalgebra is a chain, it does not have any non-trivial isomorphisms. Since $\neg d = a$ we know that $d$ generates the entire algebra, so there are no more subalgebras to consider.   

$R^{6,3}_{1,9}$: Again, there is only one possible candidate for an automorphism of this algebra, namely the bijection $g$ exchanging $b$ and $c$. This is not a homomorphism because $g(b^2) = g(0) = 0$ while $g(b)^2 = c^2 = d$. The only other subalgebra except $\{ 0,1 \}$ is $\langle a \rangle = \{ 0,1,a,b,d \}$ since $\neg a = b$,  $\neg b = \neg d = a$ and $a\rightarrow b = a \rightarrow d = b\rightarrow d = a$. This subalgebra has no non-trivial isomorphisms because it is a chain. Since $c^2 = d$, the element $c$ generates the entire algebra.  
\end{proof} 

\section*{Acknowledgments}
The second author is supported by the Luxembourg National Research Fund under the project  PRIDE17/12246620/GPS. 
   
\bibliographystyle{acm} 
\bibliography{References}  

\end{document}